\documentclass[11pt]{article}
\usepackage[square, numbers]{natbib}

\usepackage{amsmath,amssymb}

\usepackage[left=1in,top=1in,right=1in,bottom=1in]{geometry}
\usepackage{graphics,epsfig}

\usepackage{url}

\begin{document}

\newenvironment {proof}{{\noindent\bf Proof.}}{\hfill $\Box$ \medskip}

\newtheorem{theorem}{Theorem}[section]
\newtheorem{lemma}[theorem]{Lemma}
\newtheorem{condition}[theorem]{Condition}
\newtheorem{definition}[theorem]{Definition}
\newtheorem{proposition}[theorem]{Proposition}
\newtheorem{remark}[theorem]{Remark} 
\newtheorem{hypothesis}[theorem]{Hypothesis}
\newtheorem{corollary}[theorem]{Corollary}
\newtheorem{example}[theorem]{Example}

\renewcommand {\theequation}{\arabic{section}.\arabic{equation}}
\def \non{{\nonumber}}
\def \tilde{\widetilde}
\def \hat{\widehat}
\def \bar{\overline}
\newcounter{fig}
\newcommand{\ddt}{\frac{d}{dt}}
\newcommand{\lskip}{\phantom{.}}

\newcommand{\f}{\frac}

\newcommand{\Rt}{\longrightarrow}
\newcommand{\rt}{\rightarrow}
\newcommand{\RT}{\Rightarrow}
\newcommand{\Lt}{\longleftarrow}
\newcommand{\lt}{\leftarrow}
\newcommand{\LT}{\Leftarrow}
\newcommand{\LRT}{\Leftrightarrow}

\newcommand{\np}{\noindent}
\newcommand{\T}{\text}
\newcommand{\hs}{\hspace}
\newcommand{\vs}{\vspace}

\newcommand{\w}{\omega}
\newcommand{\W}{\Omega}
\newcommand{\g}{\gamma}
\newcommand{\G}{\Gamma}
\newcommand{\e}{\epsilon}
\newcommand{\oo}{\infty}
\newcommand{\s}{\sigma}
\def\SC{\mathcal}
\def\ub{\underbar}
\def \triple|{|\! | \! |}

\def\le{\left}
\def\ri{\right}

\def\l{\lambda}
\def\L{\Lambda}
\def\ot{\otimes}
\def\ott{\hat{\otimes}_\pi}
\def\oto{\hat{\otimes}_{op}}
\def\oth{\hat{\ot}_{HS}}
\def\<{\langle}
\def\>{\rangle}
\def\~{\tilde}
\newcommand{\EE}{\ensuremath{\mathsf{E}}}
\newcommand{\PP}{\ensuremath{\mathsf{P}}}

\def\y{\mathbf y^*}
\def\N{\mathbb N}
\def\Z{\mathbb Z}
\def\R{\mathbb R}
\def\C{\mathbb C}
\def\H{\mathbb H}
\def\LL{\mathbb L}
\def\K{\mathbb K}
\def\X{\mathbb X}
\def\Y{\mathbb Y}


\title{\Large\ {\bf Wong-Zakai type convergence in infinite dimensions }}

\author{Arnab Ganguly\\
ETH Zurich\\
gangulya@control.ee.ethz.ch} 
\date{}
\maketitle
\begin{abstract}
\noindent
The paper deals with convergence of solutions of a class of stochastic differential equations driven by infinite-dimensional semimartingales. The infinite-dimensional semimartingales considered in the paper are Hilbert-space valued. The theorems presented generalize the convergence result obtained by Wong and Zakai for stochastic differential equations driven by linear interpolations of a finite-dimensional Brownian motion. In particular, a general form of the correction factor is derived. Examples are given illustrating the use of the theorems to obtain other kinds of approximation results.\\


\end{abstract}

\setcounter{equation}{0}
\section{Introduction}

The subject of stochastic differential equations (SDEs) in infinite-dimensional spaces has gained substantial popularity since the publication of It{\^o}'s monograph \cite{Ito84} and Walsh's notes on stochastic partial differential equations \cite{Walsh86}. The practical applications of infinite-dimensional stochastic analysis involve investigation of various problems in a variety of disciplines including neurophysiology, chemical reaction systems, infinite particle systems, turbulence etc. 

The stability of stochastic integrals and stochastic differential equations is an important topic in stochastic analysis. More precisely, appropriate conditions on the driving sequence of semimartingales $\{Y_n\}$ are sought, such that $(X_n, Y_n) \RT (X, Y)$ will imply $X_{n-}\cdot Y_n \RT X_-\cdot Y$. Here and throughout the rest of the paper, `$\RT$' will denote convergence in distribution and $X_-\cdot Y \equiv \int X(s-) dY(s)$ is the stochastic integral of $X$ with respect to the integrator $Y$. That it is not true automatically, is shown by Wong and Zakai in \cite{WZ65_1, WZ65_2}. 
Let $W$ be a standard Brownian motion, and $W_n$  a linear interpolation of $W$ defined by
$$\f{d}{dt}W_n(t) = n\le(W(\f{k+1}{n}) -W(\f{k}{n})\ri), \quad \f{k}{n} \leq t < \f{k+1}{n}.$$
Then 
$$\int_0^t W_n(s)  \ dW_n(s) \rt \int_0^t W(s)  \ dW(s) + t/2. $$
Moreover, if $X_n$ satisfies
\begin{align}
\label{WZ0}
dX_n(t) = \sigma(X_n(t))dW_n(t) + b(X_n(t))dt ,
\end{align}
then $\{X_n\}$ does not converge to the solution of the corresponding It{\^o} SDE driven by $W$ but  goes to the solution of
\begin{align}
\label{WZ02}
dX(t)  = \sigma(X(t)) dW(t) + (b(X(t))+\f{1}{2}\sigma(X(t))\sigma'(X(t))) \ dt.
\end{align}

Generalization of the Wong-Zakai result to the multi-dimensional case  has been done by Stroock and Varadhan in \cite{SV72}. Further generalizations included replacement of the Brownian motion with general semimartingales. For continuous semimartingale differentials, Nakao and Yamato \cite{NY78} proved the following result.
\begin{theorem}
\label{WZ03}
Let $U$ be a continuous semimartingale. Suppose $X_n$ satisfies
$$dX_n(t) = \sigma(t,X_n(t),U_n(t)) dU_n(t),$$
where the  $U_n$ are piecewise $C^1$ approximations of $U$. If $U_n$ tends to $U$, then under suitable assumptions $X_n$ goes to $X$, where $X$ satisfies
$$dX(t) = \sigma(t,X(t),U(t))\ dU(t) + \f{1}{2}(\sigma \ \partial_2 \s \ + \partial_3 \s)(t, X(t), U(t)) d[U,U]_t.$$
Here $\partial_i \s$ denotes partial derivative of $\s$ with respect to the $i$-th component. 
\end{theorem}
Several extensions of the above theorem were made (see Marcus \cite{Marcus80}, Konecny \cite{Konecny83}, Protter \cite{Protter85}), where the requirement of continuous differentials was removed, and the coefficient $\s$ was allowed to be more general.

In the infinite-dimensional case, generalizations are known for approximations of some stochastic evolution equations, where the driving Brownian motion is finite dimensional, but the state-space of the solution of the SDE is infinite dimensional (see \cite{AT84, BCF88, TZ06}). Twardowska \cite{Tw93} considered the case where the driving Brownian motion is Hilbert space-valued.

Conditions like uniform tightness (UT) (Jakubowski, Me{\'m}in and Pag{\`e}s \cite{JMP89}, also see Definition \ref{UT_def}) and uniform controlled variation (Kurtz and Protter \cite{KP91}) were imposed on the driving semimartingale sequence  $\{Y_n\}$ to ensure that $X_{n-}\cdot Y_n \RT X_-\cdot Y$ if $(X_n,Y_n) \RT (X, Y)$. Extensions of the notion of uniform tightness to a sequence of Hilbert space-valued semimartingales and the corresponding weak convergence theorems for stochastic integrals were proved in \cite{Ja96}. For martingale random measures,  conditions for the desired convergence were given by Cho in \cite{Cho95, Cho96}. Kurtz and Protter \cite{KP96_2} extended the notion of uniform tightness further to a sequence of $\H^\#$-semimartingales (semimartingales indexed by Banach space $\H$ satisfying certain properties) and proved limit theorems for both stochastic integrals and stochastic differential equations. These semimartingales form a broad class of infinite-dimensional semimartingales encompassing the class of most (semi)martingale random measures, Banach space-valued semimartingales, etc. Clearly, the approximations of the driving integrators discussed above are not UT.

In the finite-dimensional case, Kurtz and Protter \cite{KP91} studied weak convergence of stochastic differential equations driven by a non-UT sequence of semimartingales. Their theorem, in particular, generalized the result obtained by Wong and Zakai.  A simpler version of their theorem (Theorem 5.10, \cite{KP91}) is stated below.
\begin{theorem}
\label{WZ3}
Let $\{U_n\}$ and $\{V_n\}$ be sequences of $R$-valued semimartingales, $b:\R\Rt \R$ be continuous, $\sigma:\R \Rt \R$ be bounded with bounded first and second order derivatives. Suppose that $X_n$ satisfies
\[X_n(t) = X_n(0) + \int_0^t \sigma(X_n(s-)) dU_n(s) + \int_0^t b(X_n(s-)) dV_n(s).\] 
Write $U_n = Y_n + Z_n$.
Denote
\begin{align*}
H_n(t)& = \int_0^tZ_n(s-) dZ_n(s) \\
K_n(t)& =  [Y_n,Z_n]_t.
\end{align*}
Assume that $ \{Y_n\}, \{H_n\}$ and $\{V_n\}$ are UT, and
$$A_n\equiv (X_n(0),V_n, Y_n,Z_n,H_n,K_n) \RT (X_0,V, Y,0,H,K)\equiv A$$
Then $(A_n,X_n)$ is relatively compact and any limit point $(A,X)$ satisfies
\begin{align*}
X(t) &= X_0 + \int_0^t \sigma(X(s-)) dY(s)  + \int_0^t  \sigma'(X(s-))\sigma(X(s-)) d(H(s) - K(s))\\
& \hspace{.4cm} + \int_0^t b(X(s-)) dV(s).
\end{align*}
\end{theorem} 

Notice that in the original Wong-Zakai case $U_n(t) = W_n(t), \ V_n(t) =t, \  Y_n(t) = W([nt+1]/n)$ and 
$Z_n(t) = W_n(t) -W([nt+1]/n $.
It could easily be proved that $\{Y_n\}$ and $\{H_n\}$ satisfy the condition of Theorem \ref{WZ3} and $(H(t) - K(t)) = t/2$. Similarly, Theorem \ref{WZ03} can be derived from Theorem \ref{WZ3} by writing $U_n =Y_n +Z_n$ for suitable $Y_n$ and $Z_n$ (see Example \ref{NYex} for a generalization). \\

The objective of the present paper is to study weak convergence of stochastic differential equations driven by infinite-dimensional semimartingales. The results obtained in this paper will be useful to investigate a broader class of approximation results. In particular, such approximation results are helpful in deriving continuous time models as limiting cases of discrete-time ones. We believe that our paper is a  step towards a unified theory of weak convergence of infinite-dimensional stochastic differential equations.

The sequence of stochastic differential equations considered in this paper are driven by  Hilbert space-valued semimartingales. However, the limiting semimartingale need not be Hilbert space-valued. The rest of the paper is structured as follows. In Section \ref{chap:Hsharp}, we discuss briefly infinite-dimensional semimartingales focussing mainly on the concept of $\H^\#$-semimartingales and Banach space-valued semimartingales. In particular, it is shown that stochastic integrals with respect to Banach space-valued semimartingales are special cases of integrals with respect to appropriate $\H^\#$-semimartingales. The main reason for doing this is to pave the way for usage of results from \cite{KP96_2} which are proven in the context of $\H^\#$-semimartingales. 
Section \ref{weakreview} is devoted to the review of the concept of uniform tightness  and weak convergence results that serve as prerequisites for our proof. Section \ref{lemmas} contains technical lemmas that are required later. The main results are presented in Section \ref{WZmain}. Theorem \ref{sde_approx} treats the case when the SDE is driven by infinite-dimensional semimartingales, but the solutions are finite-dimensional, while Theorem \ref{sde_approx_inf} extends the result to the case when the solutions of the SDE are also infinite-dimensional. The section ends with illustrative examples. A few required facts about tensor product are collected in the Appendix.

\section{Infinite-dimensional semimartingales}
\label{chap:Hsharp}
Infinite-dimensional stochastic analysis is an active research area and depending on the need, different types of infinite-dimensional semimartingales are used in modeling. A few popular notions of infinite-dimensional semimartingales include  {\it orthogonal martingale random measure}\cite{GS79}, {\em worthy martingale random measures} \cite{Walsh86}, {\em Banach space-valued semimartingales} \cite{MP80}, {\em nuclear space-valued semimartingales} \cite{Us82}. In  \cite{KP96_2}, Kurtz and Protter introduced the notion of {\it standard}  $\H^\#$-semimartingale. {\it Standard} $\H^\#$-semimartingales form a  very general class of infinite-dimensional semimartingales which includes {\it Banach space valued-semimartingales, cylindrical Brownian motion} and most semimartingale random measures. In particular,  they cover  the two important cases: space-time Gaussian  white noise and Poisson random measures. A few facts about $\H^\#$-semimartingales will be used in the present paper, and below we give a brief outline of $\H^\#$-semimartingales. 

\subsection{$\H^\#$-semimartingale}
Let $\H$ be a separable Banach space.
\begin{definition}
 An $\R$-valued stochastic process $Y$ indexed by $\H\times [0,\infty)$ is an {\bf $\H^\#$-semimartingale} with respect to the filtration $\left\{\SC{F}_t\right\}$ if
\begin{itemize}
 \item for each $h \in \H$, $Y(h,\cdot)$ is a cadlag $\left\{\SC{F}_t\right\}$-semimartingale, with $Y(h,0) = 0$;
\item for each $t>0$,  $h_1,\hdots,h_m \in \H$ and $a_1,\hdots, a_m \in \R$, we have 
$$Y(\sum_{i=1}^m a_ih_i,t) = \sum_{i=1}^m a_i Y( h_i,t) \ \ \T{a.s}.$$
\end{itemize}
\end{definition}

As in almost all integration theory, the first step is to define the stochastic integral in a canonical way for simple functions and then extend it to a broader class of integrands.

Let $Z$ be an $\H$-valued cadlag process of the form 
\begin{equation}
\label{simp}
Z(t) = \sum_{k=1}^m \xi_k(t)h_k, \quad h_1,\hdots,h_k \in \H,
\end{equation}
where the $\xi_k$ are  $\left\{\SC{F}_t\right\}$-adapted real-valued cadlag processes.\\
The stochastic integral $Z_-\cdot Y$ is defined as
$$Z_-\cdot Y(t) = \sum_{k=1}^m \int_0^t \xi_{k}(s-) d Y(h_k,s).$$
Note that the integral above is just a real-valued process. It is necessary to impose more conditions on the $\H^\#$-semimartingale $Y$ to broaden the class of integrands $Z$.

\np
Let $\SC{S}$ be the collection of all processes of the form (\ref{simp}).  Define
\begin{equation}
\label{simp_col}
\mathcal{H}_t =\left\{\sup_{s\leq t}|Z_-\cdot Y(s)|:Z \in \SC{S}, \sup_{s\leq t}\|Z(s)\| \leq 1 \right\}.
\end{equation}

\begin{definition}\label{hstd}
An $\H^\#$-semimartingale $Y$ is {\bf standard} if for each $t>0$, $\SC{H}_t$ is stochastically bounded, that is, for every $t>0$ and $\e>0$, there exists $k(t,\e)$ such that
$$P\left[\sup_{s\leq t}|Z_-\cdot Y(s)| \geq k(t,\e)\right] \leq \e$$
for all $Z \in \SC{S}$ satisfying $\sup_{s\leq t}\|Z(s)\| \leq 1$. 
\end{definition}

The extension of the stochastic integral is then achieved by approximating the integrand $X$ by processes of the form \eqref{simp}. More precisely, 

\begin{theorem}
\label{SIdef}
 Let $Y$ be a standard $\H^\#$-semimartingale, and  $X$ an $\H$-valued adapted and cadlag process. Then for every $\epsilon>0$, there exists a process $X^\epsilon$ such that $\|X(t)-X^\epsilon(t)\|<\e$, and moreover
$$X_-\cdot Y \equiv \lim_{\e\rt0}X^\e_-\cdot Y$$
exists in the sense that for each $\eta >0, t>0$,
$$\lim_{\e\rt0}P\left[\sup_{s\leq t}|X^\e_-\cdot Y(s) - X_-\cdot Y(s)|>\eta\right] = 0.$$
$X_-\cdot Y$ is a cadlag process and is defined to be the stochastic integral of $X$ with respect to $Y$
\end{theorem}

\begin{example}
\label{whitenoise}
{\rm
Let $(U,r)$ be a complete, separable metric space and $\mu$  a sigma finite measure on $(U,\SC{B}(U))$. Denote the Lebesgue measure on $[0,\infty)$ by $\lambda$, and let $W$ be a space-time Gaussian white noise on $U\times[0,\infty)$ based on $\mu\ot\lambda$, that is, $W$ is a Gaussian process indexed by $\SC{B}(U)\times [0,\infty)$
with $E(W(A,t)) = 0$ and $E(W(A,t)W(B,s)) = \mu(A\cap B) \min\left\{t,s\right\}$. For $h\in L^2(\mu)$, define
$W(h,t) = \int_{U\times [0,t)} h(x) W(dx, ds).$
The above integration is defined (see \cite{Walsh86}), and it follows that $W$ is an $\H^\#$-semimartingale with $\H=L^2(\mu)$. It is also easy to check that $W$ is standard in the sense of Definition \ref{hstd}.

}
\end{example}

\begin{example}
\label{poirndmeas} 
{\rm
Let $U,r, \mu$ and $\lambda$ be as before. Let $\xi$ be a Poisson random measure on $U\times [0,\infty)$ with mean measure $\mu\ot \lambda$, that is, 
for each $\Gamma \in \SC{B}(U)\ot \SC{B}([0,\infty))$, $\xi(\Gamma)$ is a Poisson random variable with mean $\mu\ot \lambda(\Gamma)$, and for disjoint $\Gamma_1$ and $\Gamma_2$,
$\xi(\Gamma_1)$ and $\xi(\Gamma_1)$ are independent. For $A \in \SC{B}(U)$, define
$\tilde{\xi}(A,t) = \xi(A\times \left[0,t\right]) - t\mu(A)$.
For $h \in L^2(\mu)$, let $\tilde{\xi}(h,t) = \int_{U\times [0,t)} h(x) \tilde{\xi}(dx, ds)$ and for $h \in L^1(\mu)$, let $\xi(h,t) = \int_{U\times [0,t)} h(x)\xi(dx, ds)$.
Then $\tilde{\xi}$ is a standard $\H^\#$-martingale with $\H=L^2(\mu)$ and $\xi$ is a standard $\H^\#$-semimartingale with $\H =L^1(\mu)$.

}
\end{example} 

\begin{remark}
In fact, it can be shown that most worthy martingale random measures or more generally semimartingale random measures are standard $\H^\#$-semimartingales for appropriate choices of indexing space $\H$ (see \cite{KP96_2}).
\end{remark}

\subsection {$(\LL, \hat{\H})^\#$-semimartingale and infinite-dimensional stochastic integrals}
\label{LH_sharp}
In the previous part, observe that the stochastic integrals with respect to infinite-dimensional standard $\H^\#$-semimartingales are real-valued.  Function valued stochastic integrals are of interest in many areas of infinite-dimensional stochastic analysis, for example, stochastic partial differential equations. With that in mind, we want to study stochastic integrals taking values in some infinite-dimensional space. If $Y$ is a standard $\H^\#$-semimartingale, we could put 
$H(x,t) = X(\cdot -, x)\cdot Y(t)$
where for each $x$ in a Polish space $E$, $X(\cdot,x)$ is a cadlag process with values in $\H$. The above integral is defined, but the function properties of $H$ are not immediately clear. Hence, a careful approach is needed for constructing infinite-dimensional stochastic integrals. In \cite{KP96_2},  Kurtz and Protter introduced the concept of $(\LL, \hat{\H})^\#$-semimartingale as a natural analogue of  the  $\H^\#$-semimartingale for  developing infinite-dimensional stochastic integrals. Below, we give a brief outline of that theory.\\

Let $(E,r_E)$ and $(U,r_U)$ be two complete, separable metric spaces. Let  $\LL, \H $ be separable Banach spaces of $\R$-valued functions on $E$ and $U$ respectively. Note that for function spaces, the product $fg, f \in \LL, g \in \H$ has the natural interpretation of point-wise product. Suppose that  $\left\{f_i\right\}$ and $\left\{g_j\right\}$  are such that the finite linear combinations of the $f_i$  and the finite linear combinations of the $g_j$ are dense in $\LL$ and $\H$ respectively. 
\begin{definition}
\label{Hhat}
Let $\hat{\H}$ be the completion of the linear space $\left\{\sum_{i=1}^l\sum_{j=1}^ma_{ij}f_ig_j: f_i \in \left\{f_i\right\}, g_j \in \left\{g_j\right\}\right\}$ with respect to some norm $\|\cdot\|_{\hat{\H}}$.
\end{definition}

For example, if
$$\|\sum_{i=1}^l\sum_{j=1}^ma_{ij}f_ig_j\|_{\hat{\H}} = \sup\left\{\sum_{i=1}^l\sum_{j=1}^ma_{ij}\<\l,f_i\>\<\eta,g_j\>: \l \in \LL^*, \eta \in \H^*, \|\l\|_{\LL^* }\leq 1, \|\eta\|_{\H^*} \leq 1\right\}$$
then $\hat{\H}$ can be interpreted as a subspace of  the space of bounded operators, $L(\K^*,\LL)$.\\

Let  $\SC{S}_{\hat{\H}}$  denote the space of all processes $X \in D_{\hat{\H}}[0,\infty)$ of the form
\begin{align}\label{simLH}
X(t)= \sum_{ij}\xi_{ij}(t)f_ig_j,
\end{align}
where the $\xi_{ij}$ are $\R$-valued, cadlag, adapted processes and only fintely many $\xi_{ij}$ are non zero.
For $X \in \SC{S}_{\hat{\H}}$, define
$$X_-\cdot Y(t) = \sum_i f_i \sum_j \int_0^t \xi_{ij}(s-)\ dY(g_j,s).$$
Notice that $X_-\cdot Y \in D_{\LL}[0,\infty).$

\begin{definition}\label{std2}
An $\H^\#$-semimartingale is a {\bf standard $(\LL, \hat{\H})^\#$-semimartingale} if
$$\SC{H}_t \equiv \left\{\sup_{s\leq t}\|X_-\cdot Y(s)\|_\LL: X \in \SC{S}_{\hat{\H}}, \sup_{s\leq t}\|X(s)\|_{\hat{\H}} \leq 1\right\}$$
is stochastically bounded for each $t > 0$.
\end{definition}

As in Theorem \ref{SIdef}, under the standardness assumption, the definition of $X_-\cdot Y$ can be extended to all cadlag $\hat{\H}$-valued processes $X$, by approximating $X$ by a sequence of processes of the form \eqref{simLH}. 

\begin{remark} The standardness condition in Definition \ref{std2} will follow if there exists a constant $C(t)$ such that
$$E\left[\|X_-\cdot Y(t)\|_\LL\right] \leq C(t), \quad t>0$$
for all $X \in \SC{S}_{\hat{\H}}$ satisfying $\sup_{s\leq t}\|X(s)\|_{\hat{\H}} \leq 1$.
\end{remark}

\begin{remark}
If $\H$ and $\LL$ are general Banach spaces (rather than Banach spaces of functions), then $\hat{\H}$ could be taken as the completion of $\LL\ot\H$ with respect to some norm, for example the Hilbert-Schmidt norm or the projective norm (see \cite{Rr02}).
\end{remark}

\subsection{Banach space-valued semimartingales }
Standard references for the materials in this section are \cite{MP80, Mm82}.
We start with the definition of martingales taking values in a Banach space $\H$. The definition is analogous to that of real-valued martingales.
\begin{definition}
Let $(\Omega, \SC{F}, \{\SC{F}_t\}, P)$ be a complete probability space. A stochastic process  $M$ taking values in $\H$ is  an {\bf $\{\SC{F}_t\}$-martingale} if
\begin{itemize}
 \item $M$ is  $\{\SC{F}_t\}$-adapted;
\item $E\|M_t\|_\H < \infty$, for all $t>0$;
\item  for every $F \in \SC{F}_s$,\ $\int_F M_t \ dP = \int_F M_s \ dP,$ where $t>s>0$.
\end{itemize}
The integration above is in the Bochner sense.
 \end{definition}

\begin{remark}
For every $h \in \H$ and $h^* \in \H^* $, let $\<h,h^*\>_{\H,\H^*}$ be defined by
\begin{align}
 \label{opaction}
\<h,h^*\>_{\H,\H^*} = h^*(h) = \<h^*,h\>_{\H^*,\H}.
\end{align}
Note that if $M$ is an $\H$-valued martingale, then $\<M(t),h^*\>_{\H,\H^*}$ is a real-valued martingale for every $h^* \in \H^*$. Conversely, 
if $\H$ is separable and $E\|M_t\|_{\H} < \infty$ for every $t>0$, then  $M$ is an $\H$-valued martingale if for every $h^*\in \H^*$, $\<M(t),h^*\>_{\H,\H^*}$ is a real-valued martingale. This is because for separable Banach spaces, the notion of Bochner integral coincides with that of Petis integral.
\end{remark}

Just like the real-valued case, the notion of martingales can be generalized to that of local martingales. Below we define Banach space-valued  semimartingales
\begin{definition}
Let $(\Omega, \SC{F}, \{\SC{F}_t\}, P)$ be a complete probability space. A stochastic process  $Y$ taking values in $\H$ is  an {\bf $\{\SC{F}_t\}$-semimartingale} if $Y$ could be decomposed into
$$Y = M + V,$$
where $M$ is a local martingale, and $V$ is a finite variation process on every bounded interval $[0,t] \subset [0,\infty).$
\end{definition}
\begin{remark}
The local martingale $M$ in the above decomposition can be taken as locally square integrable (see \citep[Theorem 23.6]{Mm82} and  \citep[Section 9.16]{MP80} ). In fact, M{\'e}tivier defined semimartingale when the local martingale part is locally square integrable.
\end{remark}

\subsection{Integration with Banach space-valued semimartingales}
Let $X$ be an  $\{\SC{F}_t\}$-adapted, cadlag process taking values in $\H^*$. Suppose that $Y$ is  an $\H$-valued $\{\SC{F}_t\}$-adapted semimartingale. Let $\sigma =\{t_i\}$ be a partition of $[0,\infty)$. Define
\begin{align}
 \label{xsigma}
X^\sigma(s) = \sum_{i}X(t_i)1_{[t_i,t_{i+1})}(s)
\end{align}
and the stochastic integral $X^\sigma_- \cdot Y(t)$ as
$$X^\sigma_ -\cdot Y(t) = \sum_{i}\<X(t_i), Y(t_{i+1}\vee t)  -Y(t_i\vee t)\>_{\H^*,\H},$$
Notice that $X^\sigma_ -\cdot Y$ is a real-valued process.
The following theorem proves the existence of the stochastic integral 
\begin{theorem}
\label{SIdef1}
There exists an $\{\SC{F}_t\}$-adapted, real-valued cadlag process $X_-\cdot Y$ such that for all $T>0$,
$$\sup_{t\leq T}|X^\sigma_-\cdot Y(t) - X_-\cdot Y(t)| \stackrel{P}\Rt 0, \ \mbox{ as } \ \|\sigma\| \rt 0.$$
\end{theorem}

The following lemma (see \citep[Section 10.9]{MP80}) gives a  bound for the stochastic integral.
\begin{lemma}
Let $Y$ be an $\{\SC{F}_t\}$-adapted semimartingale taking values in a Banach space $\H$ and $X$ an  $\{\SC{F}_t\}$-adapted, cadlag process taking values in $\H^*$. Then, there exists a nondecreasing, $\{\SC{F}_t\}$-adapted, real-valued cadlag process $Q$
such that
\begin{align}
\label{SI_bound1}
E[\sup_{t\leq T}|X_-\cdot Y(s)|^2] \leq E[\int_0^T \|X_{s-}\|^2_\H dQ_s]
\end{align}
Integration in the right side is in the Riemann-Stieltjes sense.
\end{lemma}

\subsubsection{Banach space-valued semimartingale as standard $\H^\#$-semimartingale}
Let $Y$ be a semimartingale taking values in a Banach space $\K$. We will show that $Y$ can be considered as an $\H^\#$-semimartingale, with $\H = \K^*$.
Since $\K$ is isometrically embedded in $\K^{**}$, consider $Y$ as an element of $\K^{**}$. Then notice that 
\begin{itemize}
\item for each $h \in \K^* $, $Y(h,\cdot) \equiv \<Y(t),h\>_{\K,\K^*}$ is a real-valued semimartingale;
\item for $h_1, h_2 \in \K^*$, $Y(h_1+h_2,\cdot) = Y(h_1,\cdot)+Y(h_2,\cdot)$.
\end{itemize}
This proves that $Y$ is an $\H^\#$-semimartingale with $\H = \K^*$, and now (\ref{SI_bound1}) proves that $Y$ is standard. It is obvious that the two definitions of stochastic integral (see Theorem \ref{SIdef} and Theorem \ref{SIdef1}) coincide.

\begin{remark}
If \ $\K = \LL^*$, for some Banach space $\LL$, then $Y$ can be considered as an $\LL^\#$-semimartingale.
\end{remark}

\subsubsection{Hilbert space-valued stochastic integrals}
As before, let $Y$ be a semimartingale taking values in a Banach space $\K$. 
Let $\LL$ be a separable Hilbert space. Let $X$ be an  $\{\SC{F}_t\}$-adapted, cadlag process taking values in the operator space, $L(\K,\LL)$. Let $\sigma =\{t_i\}$ be a partition of $[0,\infty)$. Define
\begin{align}
 \label{xsigma2}
X^\sigma(s) = \sum_{i}X(t_i)1_{[t_i,t_{i+1})}(s)
\end{align}
and the stochastic integral $X^\sigma_- \cdot Y(t)$ as
$$X^\sigma_ -\cdot Y(t) = \sum_{i}X(t_i)(Y(t_{i+1}\wedge t)  -Y(t_i\wedge t)).$$
Notice that $X^\sigma_ -\cdot Y$ is an $\LL$-valued process.
The following theorem proves the existence of the stochastic integral. 
\begin{theorem}
\label{SIdef2}
There exists an $\{\SC{F}_t\}$-adapted, $\LL$-valued cadlag process $X_-\cdot Y$, such that for all $T>0$,
$$\sup_{t\leq T}\|X^\sigma_-\cdot Y(t) - X_-\cdot Y(t)\|_\LL \stackrel{P}\Rt 0.$$
\end{theorem}
Similar to (\ref{SI_bound1}), we have:
\begin{lemma}
Let $Y$ be an $\{\SC{F}_t\}$-adapted semimartingale taking values in a Banach space $\K$. Then, there exists a nondecreasing, $\{\SC{F}_t\}$-adapted, real-valued cadlag process $Q$, such that for any Hilbert space $\LL$
\begin{align}
\label{SI_bound2}
E[\sup_{t\leq T}\|X_-\cdot Y(s)\|_{\LL}^2] \leq E[\int_0^T \|X_{s-}\|^2_{op}dQ_s],
\end{align}
whenever $X$ is an  $\{\SC{F}_t\}$-adapted, cadlag $L(\K,\LL)$-valued process. Here $\|\cdot\|_{op}$ denotes the operator norm.
\end{lemma}
See  \citep[Section 10.9, Section 6.7]{MP80})

\begin{remark}
The above lemma might not be true if $\LL$ is an arbitrary Banach space.
\end{remark}

\begin{remark}
\label{lhsharp}
If $Y$ is a $\K$-valued semimartingale, then (\ref{SI_bound2}) shows that for any Hilbert space $\LL$, $Y$ can be considered as a standard $(\LL,\hat{\H})^\#$-semimartingale. Here  $\hat{\H}$ is the completion of the space $\LL\ot\K^*$ with respect to some norm which makes $\LL\ot\K^* \subset L(\K,\LL)$.
\end{remark}

Suppose that $X$ and $ Y$ are two cadlag semimartingales taking values in $\K,\K^*$. Then both $X_-\cdot Y$ and $Y_-\cdot X$ are defined. We define the (scalar) covariation process $[X,Y]$ as 
\begin{align}
\label{intpart1}
[X,Y]_t = \<X(t),Y(t)\>_{\K,\K^*}- \<X(0),Y(0)\>_{\K,\K^*}- X_-\cdot Y(t) -Y_-\cdot X(t). 
\end{align}
It is easy to see that 
$$[X,Y]_t = \lim_{\|\sigma\| \rt 0} \sum_i\<X(t_{i+1}) - X(t_i),Y(t_{i+1}) - Y(t_i)\>_{\K,\K^*} $$
where $\sigma = \{t_i\}$ is a partition of $[0,t]$, and $\|\sigma\| = \sup(t_{i+1}-t_i)$ is the mesh of the partition $\sigma$.

\subsection{Tensor stochastic integration}
\label{tenSI}
We briefly outline the  theory of tensor stochastic integration. It will be used in the next chapter. The reader might want to look at Section \ref{tenprod} before reading this part. We assume that $Y$ is an adapted $\K$-valued semimartingale, where $\K$ is a separable Hilbert space with inner product denoted by $\<\cdot,\cdot\>_\K$.
Let $X$ be a cadlag and adapted $\K$-valued process.
The tensor stochastic integral $\int X_-\ot dY$ is defined as
$$\int_0^t X(s-)\ot dY(s) = \lim_{\|\sigma\| \rt 0}\sum_{i} X(t_i)\ot (Y(t_{i+1})-Y(t_i)),$$
where $\sigma = \{t_i\}$ is a partition of $[0,t]$, and $\|\sigma\| = \sup(t_{i+1}-t_i)$ is the mesh of the partition $\sigma$.

\begin{theorem}
$\lim_{\|\sigma\| \rt 0}\sum_i X(t_i)\ot (Y(t_{i+1})-Y(t_i))$ exists.
\end{theorem}

\begin{proof}
Below, we give a quick proof which illustrates the fact that the tensor integration is an example of stochastic integration with respect to a standard 
$(\LL,\hat{\H})^\#$-semimartingale, for appropriate $\LL$ and $\hat{\H}$.
Take $\LL = \K\hat{\ot}_{HS}\K$, the completion of the space $\K\ot\K$ with respect to the Hilbert-Schmidt norm (see \ref{HSnorm}). Recall that $\K\hat{\ot}_{HS}\K$ is a Hilbert space and can be identified with the space of Hilbert-Schmidt operators $HS(\K,\K)$.
Let $\SC{H} = \LL\ot\K$, that is, $\SC{H}$ is  the space of all elements of the form:
$$\sum_{i,j=1}^{I,J}c_{ij} \l_i \ot k_j, \ \ \l_i \in \LL, k_j \in \K, c_{ij} \in \R.$$
Consider $\SC{H}$ as a subspace of $L(\K, \LL)$, by defining the action of an element in $\SC{H}$ on $\K$ as 
$$\sum_{i,j=1}^{I,J}c_{ij} \l_i \ot k_j(k) = \sum_{i,j=1}^{I,J}c_{ij}\< k_j,k\>_\K \l_i, \ \ k \in \K. $$
Let $\hat{\H}$ be the completion of the space $\SC{H}$ with respect to the operator norm. Suppose that $\{e_i\}$ forms an orthonormal basis of $\K$.
For $h \in \K$, define
$$\hat{h} = \sum_{i,j}\<h,e_i\>_\K(e_i\ot e_j)\ot e_j$$
so that for any $g\in \K$,
$$\hat{h}(g) = \sum_{i,j}\<h,e_i\>_\K\<g,e_j\>_\K e_i\ot e_j .$$
Observe that 
$$\hat{h}(g) = h\ot g.$$
It is now trivial to check that $\hat{h} \in \hat{\H}$, and  $h\rt \hat{h}$ is an isometric isomorphism from $\K$ into $\hat{\H}$. Consequently, $h$ can be identified with $\hat{h}$ and thought of as an element of $\hat{\H}$. 
Therefore,
\begin{align*}
 \sum_{i} X(t_i)\ot (Y(t_{i+1})-Y(t_i)) = \int_0^tX^{\sigma}(s-) \ot dY(s) = \int_0^t \hat{X}^{\sigma}(s-)dY(s). 
\end{align*}
The last quantity has a limit as $\|\sigma\| \rt 0$, because $Y$ is a standard $(\LL,\hat{\H})^\#$-semimartingale, for any Hilbet space $\LL$ (see Remark \ref{lhsharp}).
\end{proof}

Note that by the construction, $\int X_-\ot dY \in  \K\oth\K = HS(\K,\K)$. Since the tensor product is not usually symmetric, $\int X_-\ot dY \neq \int dY\ot X_- $. But  as Lemma \ref{obs1} shows, we have the following relation
$$(\int X_-\ot dY)^* = \int dY\ot X_-,$$
where $*$ denotes the operator adjoint.

\begin{lemma}
\label{obs1}
Let $X$ be an adapted, cadlag $\K$-valued process and $Y$ an adapted $\K$-valued semimartingale.
\begin{align*}
\<\int_0^tX(s-)\ot dY(s)\phi_k,\psi_k\>_\K &= \<\int_0^tX(s-)\ot dY(s),\phi_k\ot\psi_k\>_{ \K\hat{\ot}_{HS}\K} \\
& = \int_0^t \<X(s-),\phi_k\>_\K \ d\<Y(s),\psi_k\>_\K \\
\<\int_0^tdY(s)\ot X(s-)\psi_k,\phi_k\>_\K & =\<\int_0^tdY(s)\ot X(s-), \psi_k\ot \phi_k\>_{\K\hat{\ot}_{HS}\K}\\
& = \int_0^t \<X(s-),\phi_k\>_\K \ d\<Y(s),\psi_k\>_\K.
\end{align*}

\end{lemma}
\begin{proof} Let $\s$ denote the partition $\{t_i\}$ of $[0,t]$, and denote $X^\sigma$ by (\ref{xsigma})
Notice that
\begin{align*}
\<\int_0^tX^\s_s\ot dY_s , \phi_k\ot\psi_k\>_{ \K\hat{\ot}_{HS}\K} & = \sum_i \<(X(t_i))\ot(Y(t_{i+1}) - Y(t_i)),\phi_k\ot\psi_k\>_{ \K\hat{\ot}_{HS}\K} \\
& = \sum_i\<X(t_i),\phi_k\>_\K \ \<Y(t_{i+1}) - Y(t_i),\psi_k\>_\K\\
& = \int_0^t \<X^\s_s,\phi_k\>_\K \ d\<Y_s,\psi_k\>_\K.
\end{align*}
The theorem follows by taking limit as $\|\sigma\| \rt 0$, and using the continuity of the inner product.
The second part is similar.
\end{proof}

Define $Z =\int X_-\ot dY$. Since $Z \in \K\oth\K =HS(\K,\K) $, by Remark \ref{lhsharp},  $Z$ is a standard $(\LL, \hat{\H})^\#$-semimartingale, where  $\LL$ is any  Hilbert space and $\hat{\H}$ is the completion of the space $\LL\ot (\K\oth\K)$ with respect to some norm such that $\hat{\H} \subset L(\K\oth\K, \LL)$.
Hence, if $J$ is an $\hat{\H}$-valued cadlag and adpated  process, the stochastic integral $J_-\cdot Z$ is defined.

Recall that for any two Hilbert spaces $\X, \Y$, $\X\oth\Y=HS(\Y,\X)\subset L(\Y,\X)$ (see \ref{HSnorm}). In particular,  for $u=\sum_{i=1}^mx_i\ot y_i$ and $y\in \Y$, $u(y) = \sum_ix_i \<y,y_i\>$.
Note that  $\|u\|_{op} \leq \|u\|_{HS}$, where $\|\cdot\|_{op}$ denotes the operator norm. The following chain rule holds.

\begin{theorem}
\label{ch1}
Suppose $J$ is an $(\K\oth\K) $-valued cadlag and adapted  process and $Z_t = \int_0^tX(s-)\ot dY(s)$. Then
\[\int_0^tJ(s-) \ dZ(s) = \int_0^tJ(s-)(X(s-)) \ dY_s.\]
\end{theorem}

\begin{proof}
First, take $J$  of the form
\begin{align}
\label{simple}
J(s) = \sum_{k=1}^n \xi_k(s)\phi_k\ot\psi_k.
 \end{align}

Then note that
\begin{align*}
\int_0^tJ(s-) \ dZ(s)& = \sum_{k=1}^n \int_0^t \xi_k(s-) d\<Z(s),\phi_k\ot\psi_k\>\\
& = \sum_{k=1}^n \int_0^t \xi_k(s-) \<X(s-),\phi_k\> \ d\<Y(s),\psi_k\> \ \ \ (\mbox{by Lemma \ref{obs1}})\\
& =   \int_0^t \sum_{k=1}^n \xi_k(s-) \<X(s-),\phi_k\> \psi_k \ dY(s)\\
& = \int_0^t J(s-)(X(s-)) \ dY(s).
\end{align*}
The third equality follows from the definition of the stochastic integral with  respect to a standard $\K^\#$-semimartingale, and the last one by identifying 
$\K\oth\K$ with $L(\K,\K)$.
Now, for any $\K\oth\K$-valued adapted process $J$, there is a sequence of $\K\oth\K$ valued adapted processes  $J_n$ of the form (\ref{simple}) such that  $\sup_{s\leq t}\|J_n(s) -J(s)\|_{HS} \rt 0$, which in turn implies $\sup_{s\leq t}\|J_n(s) -J(s)\|_{op} \rt 0$.
Letting $n\rt \infty$ in
$$\int_0^tJ_n(s-) \ dZ(s) =  \int_0^t J_n(s-)(X(s-)) \ dY(s),$$
we are done.
\end{proof}

Similar to (\ref{intpart1}), we define the tensor covariation as 
\begin{align}
 \label{intpart2}
[X,Y]^\ot_t = X(t)\ot Y(t) - X(0)\ot Y(0) -\int_0^t X(s-)\ot dY(s) -\int_0^t dX(s-) \ot Y(s)
\end{align}
It is easy to see that 
 $$[X,Y]^\ot_t = \lim_{\|\sigma\| \rt 0} \sum_i (X(t_{i+1}) - X(t_i))\ot(Y(t_{i+1}) - Y(t_i)) $$
where $\sigma = \{t_i\}$ is a partition of $[0,t]$, and $\|\sigma\| = \sup(t_{i+1}-t_i)$ is the mesh of the partition $\sigma$.

Let $B$ be a Banach space. For $\phi \in D_{B}[0,\infty)$, define the total variation of $\phi$ in the interval $[0,t]$ as 
\begin{align}
 \label{FVdef}
T_t (\phi) = \sup_{\sigma} \sum_i \|\phi(t_i) -\phi(t_{i-1})\|_{B},
\end{align}
where as before, $\sigma=\{t_i\}$ is a partition of the interval $[0,t]$. We say $\phi$ is of locally finite variation (or sometimes simply finite variation) if $T_t(\phi) < \infty $, for all $t>0$. 

\begin{remark}\label{tensorcov}
For any $\K$-valued semimartingale $Y$, $[Y,Y]^\ot$ is an $\K\oth\K = HS(\K,\K)$-valued process. In fact, it can be shown that almost all paths of $[Y,Y]^\ot$ take values in the space of nuclear operators $\SC{N}(\K,\K)$ and $trace([Y,Y]^\ot_t) =[Y,Y]_t$. Moreover, the total variation of paths of $[Y,Y]^\ot$ in the nuclear norm (hence also in the Hilbert-Schmidt norm) satisfies $T_t([Y,Y]^\ot) \leq [Y,Y]_t$. (See \citep[Theorem 26.11]{Mm82})
\end{remark}

\section{Uniform tightness and weak convergence results}\label{weakreview}

Since the state space of the $\H^\#$-semimartingales is not known, weak convergence of a sequence of $\H^\#$-semimartingales is defined in the following way.
\begin{definition}
\label{weakcon_def}
Let $\LL$ and $\H$ be two separable Banach spaces. Let $\{Y_n\}$ be a sequence of $\{\SC{F}^n_t\}$-adapted $\H^\#$-semimartingales and $\{X_n\}$ be a sequence of cadlag, $\{\SC{F}^n_t\}$ adapted $\LL$ valued processes. $(X_n,Y_n) \RT (X,Y)$ if for every finite collection of elements  $\phi_1,\hdots \phi_d \in \H$, 
$$(X_n,Y_n(\phi_1,\cdot),\hdots,Y_n(\phi_d,\cdot)) \RT (X,Y(\phi_1,\cdot),\hdots,Y(\phi_d,\cdot))$$
in  $\ D_{\LL\times \R^d}[0,\infty)$.
\end{definition}

Let $\LL,\K$ be separable Banach spaces, and define $\hat{\H}$ to be the completion of the space $\LL\ot\K$ with respect to some norm.
Let $\{\SC{F}^n_t\}$ be a sequence of right continuous filtrations. Let $\mathcal{S}^n$ denote the space of all $\hat{\H}$-valued processes $Z$, such that $\|Z(t)\|_{\hat{\H}} \leq 1$ and is of the form
\[Z(t) = \sum_{i,j=1}^{I,J} \xi_{ij}(t) \l_i\ot h_j, \ \  \l_i \in \LL, h_j \in \K\]
where the $\xi_{ij}$ are cadlag and $\{\mathcal{F}^n_t\}$-adapted $\R$-valued processes.
\begin{definition}
\label{UT_def}
A sequence of $\{\SC{F}^n_t\}$ adapted, standard $(\LL, \hat{\H})^\#$-semimartingales  $\{Y_n\}$ is {\it uniformly tight} (UT) if,  for every $\delta > 0$ and $t>0$, there exists a $M(t,\delta)$ such that 
\begin{align}
 \sup_{Z \in \mathcal{S}^n}  P[\sup_{s\leq t}\|Z_-\cdot Y_n(s)\|_{\LL} > M(t,\delta)] \leq \delta.
\end{align}
\end{definition}
\begin{remark}
Uniform tightness of the sequence $\{Y_n\}$ would follow if, for every $t>0$, there exists a constant $C(t)$ (not depending on $n$), such that
$$ \sup_{Z \in \mathcal{S}^n}E[\sup_{s\leq t}\|Z_-\cdot Y_n(s)\|_{\LL}] \leq C(t).$$
\end{remark}

\begin{theorem} (\citep[Theorem 4.2]{KP96_2})
 \label{SI_limit}
For each $n=1,2,\hdots$, let $Y_n$ be an $\{\SC{F}^n_t\}$-adapted,  standard $(\LL, \hat{\H})^\#$-semimartingale.
Assume that the sequence $\{Y_n\}$ is UT.  If $(X_n,Y_n) \RT (X,Y)$, then there is a filtration $\{\SC{F}_t\}$ such that $Y$ is an $\{\SC{F}_t\}$-adapted, standard $(\LL, \hat{\H})^\#$-semimartingale, $X$ is $\{\SC{F}_t\}$-adapted and $(X_n,Y_n, X_{n-}\cdot Y_n) \RT (X,Y, X_-\cdot Y)$.\\
 If $(X_n,Y_n) \stackrel{P}\Rt (X,Y)$ in probability then $(X_n,Y_n, X_{n-}\cdot Y_n) \stackrel{P}\Rt (X,Y, X_-\cdot Y)$.
\end{theorem}
A similar theorem for stochastic differential equations has also been proved.
\begin{theorem} (\citep[Theorem 7.5]{KP96_2})
\label{SDE_limit}
Let $\LL = \R^d$. For each $n=1,2,\hdots$, let $Y_n$ be an $\{\SC{F}^n_t\}$-adapted,  standard $(\LL, \hat{\H})^\#$-semimartingale. Suppose that $(U_n, X_n, Y_n)$ satisfies 
$$X_n = U_n + F_n(X_{n-})\cdot Y_n,$$
where $F_n, F : \R^d \rt \K^d$ are measurable functions satisfying
\begin{itemize}
 \item $F_n \rt F$ uniformly over compact subsets of $\R^d$;
\item $F$ is  continuous;
\item $\sup_n\sup_x \|F_n(x)\|_{\K^d} < \infty$. 
\end{itemize}
If $(U_n,Y_n) \RT (U,Y)$ and $\{Y_n\}$ is UT, then $\{(U_n, X_n, Y_n)\}$ is relatively compact and any limit point $(U, X, Y)$ satisfies 
\begin{align*}
X = U + F(X_{-})\cdot Y.
\end{align*} 
\end{theorem}
The corresponding theorem for general $\LL$ is:
\begin{theorem}
 \label{SDE_limit2} (\citep[Theorem 7.6]{KP96_2})
For each $n=1,2,\hdots$, let $Y_n$ be an $\{\SC{F}^n_t\}$-adapted,  standard $(\LL, \hat{\H})^\#$-semimartingale. Suppose that $(U_n, X_n, Y_n)$ satisfies
$$X_n = U_n + F_n(X_{n-})\cdot Y_n,$$
where $F_n, F : \LL \rt \hat{\H}$ are measurable functions satisfying
\begin{itemize}
 \item $F_n \rt F$ uniformly over compact subsets of $\LL$;
\item  $F$ is  continuous;
\item $\sup_n\sup_x \|F_n(x)\|_{\hat{\H}} < \infty;$ 
\item for each $\delta>0$, there exists a compact $E_\delta$ such that $\sup_{s\leq t}\|x(s)\|_{\hat{\H}} \leq \delta$ implies that $F_n(x(t)) \in E_\delta$ for all $n$.
\end{itemize}
If $(U_n,Y_n) \RT (U,Y)$ and $\{Y_n\}$ is UT, then $\{(U_n, X_n, Y_n)\}$ is relatively compact and any limit point $(U, X, Y)$ satisfies 
\begin{align}\label{limsde1}
X = U + F(X_{-})\cdot Y.
\end{align} 
\end{theorem}

\begin{remark}
Suppose that in addition to the conditions of Theorem \ref{SDE_limit} or Theorem \ref{SDE_limit2}, strong uniqueness holds for \eqref{limsde1} for any versions of $(U,Y)$ for which $Y$ is an $\H^\#$ or $(\LL,\hat{\H})^\#$-semimartingale and that $(U_n,Y_n)\rt (U,Y)$ in probability. Then $(U_n,Y_n, X_n) \rt (U,Y,X)$ in probability.
\end{remark}


\section{A few lemmas}\label{lemmas}
\begin{lemma}\label{barmap}
Let $\H$ and $\K$ be two separable Hilbert spaces. Let $\LL = \K\oth\H = HS(\H,\K)$. Then $\H$ is continuously embedded in $L(\LL,\K)$.
\end{lemma}

\begin{proof}
For $h\in \H$, define $\bar{h} \in L(\LL,\K)$ by
$$\bar{h}(l) = l(h), \ \ \ l \in \LL.$$
Notice that $h \in \H \Rt \bar{h} \in L(\LL,\K)$ is an isomorphism and $\|\bar{h}\| \leq \|h\|_\H$.

\end{proof}

\begin{lemma}\label{tildemap}
Let $\H,\K$ and $\LL$ be as in Lemma \ref{barmap}. Let $u \in L(\K,\LL)$, $v \in HS(\H,\K)$. Define 
$\tilde{uv} \in L(\H\oth\H,\K)$ by
$$\tilde{uv}(h_1\ot h_2) = \bar{h}_1uv (h_2),$$
where $\bar{h}_1$ is as in the proof of the previous lemma. Then $\tilde{uv} \in HS(\H\oth\H,K)$ and 
$$\|\tilde{uv}\|_{ HS(\H\oth\H,K)} = \|uv\|_{HS(\H,\LL)}.$$
\end{lemma}

\begin{proof}
If $\{e_i\}$ is an orthonormal basis of $\H$, then $\{e_i\ot e_j\}$ forms an orthonormal basis for $\H\oth\H$. 
Notice that 
$$\tilde{uv}(e_i\ot e_j) = \bar{e}_iuv(e_j) = uv(e_j)(e_i).$$
It follows that
\begin{align*}
\sum_{i,j}\|\tilde{uv}(e_i\ot e_j)\|_\K^2& = \sum_j \sum_i \|uv(e_j)(e_i)|\|^2_\K = \sum_j \|uv(e_j)\|^2_\LL\\
& = \|uv\|^2_{HS(\H,\LL)}
\end{align*}
\end{proof}

\np
Notice that if $\K=\R$, then $\tilde{uv} =u\ot v$. The following lemma is a generalization of Lemma \ref{ch1}.

\begin{lemma}
Let $H$, $V$ and $U$  be  adapted cadlag processes taking values in $\H$, $\LL\equiv HS(\H,\K)$  and $L(\K,\LL)$ respectively. Let $Z$ be an adapted $\H$-valued semimartingale. Then
\begin{equation}
\label{eqch3}
\int_0^t \bar{H(s-)} U(s-)V(s-) dZ(s) = \int_0^t \tilde{U(s-)V(s-)} dR(s),
\end{equation}
where $R(t) = \int_0^t H(s-) \ot dZ(s)$. Here the \ $\tilde{}$ \ and\ \ $\bar{\phantom{d} }$\  mappings are as in Lemma \ref{tildemap} and the proof of Lemma \ref{barmap} respectively.

\end{lemma} 

\begin{proof}
Notice that both sides take values in $\K$. For  $u \in L(\K,\LL)$, $v \in HS(\H,\K)$, define $\hat{uv} \in L(\H, HS(\H,\K))$ by
$$\hat{uv}(h) = \bar{h}uv.$$
Note that $\hat{uv} \neq uv$.
It is easy to check that in fact, $\hat{uv} \in HS(\H, HS(\H,\K))$. Thus for any $\l \in L(\K,\R)$, $\l\hat{uv} \in HS(\H,HS(\H,R))$ and can be identified with $\l\tilde{uv} \in HS(\H,\H)$. 
The proof now follows by applying $\l$ on both sides of (\ref{eqch3}) and using Lemma \ref{ch1} to verify their equality.
\end{proof}

\begin{lemma}
\label{qd1}
Let $\H$ be  a separable Hilbert space and $Y$  an $\H$-valued adapted semimartingale. 
Suppose that $J$ and $ V$ are  cadlag, adapted  processes taking values in $\H$.
Define  $X =  V_-\cdot Y$. Note that $X$ is a real- valued semimartingale. 
Let  $U(t) = \int_0^t J(s-)dX(s)$. Then for any $\H$-valued adapted semimartingale $Z$, we have
\[[U,Z] = \int J_-\ot V_- \ d[Z,Y]^\ot.\]
\end{lemma}

\begin{proof}
 Let $\s = \{t_i\}$ be a partition of $[0,t]$. For a process $H$, define $H^\sigma$
by
$$H^\s(s) = \sum_{i}H(t_i)1_{[t_i,t_{i+1}]}(s).$$
Let
\[X_\s(u) \equiv \int_0^u V^\s(s-) \ dY_s = \sum_i \<V(t_i\wedge u-), Y(t_{i+1}\wedge u) - Y(t_i\wedge u)\>_\H \ \ \ \ 0\leq u \leq t\] and
\[U_\s(u) \equiv \int_0^u J^\s(s-)\ot dX_\s(s) = \sum_i J(t_i\wedge u-)\ot (X_\s(t_{i+1}\wedge u) - X_\s(t_i\wedge u)). \]
Denote
$$A = \int_0^t J(s-)\ot V(s-) d[Z,Y]^\ot(s)$$
and notice that

\begin{align*} 
A&=  \lim_{\|\s\|\rt 0} \sum_i \<J(t_i)\ot V(t_i),(Z(t_{i+1} - Z(t_i)))\ot (Y(t_{i+1}) - Y(t_i))\>_{HS}   \\
& = \lim_{\|\s\|\rt 0} \sum_i \<J(t_i), Z(t_{i+1} - Z(t_i)) \>_{\H} \  \<V(t_i),Y(t_{i+1}) - Y(t_i) \>_{\H}  \\
& = \lim_{\|\s\|\rt 0} \sum_i \<J(t_i), Z(t_{i+1} - Z(t_i)) \>_{\H} \  (V^\s_-\cdot  Y(t_{i+1}) - V^\s_-\cdot  Y(t_i))\\
& = \lim_{\|\s\|\rt 0} \sum_i \<J(t_i), Z(t_{i+1} - Z(t_i)) \>_{\H} (X_\s(t_{i+1}) - X_\s(t_i)) \\
& = \lim_{\|\s\|\rt 0} \sum_i \<J(t_i)(X_\s(t_{i+1}) - X_\s(t_i)), Z(t_{i+1} - Z(t_i)) \>_{\H} \\
& = \lim_{\|\s\|\rt 0} \sum_i \<U_\s(t_{i+1}) - U_\s(t_i), Z(t_{i+1}) - Z(t_i)\>_{\H} = [U, Z]_t.
\end{align*}
\end{proof}

If $X$ is an  $\LL \equiv HS(\H,\K)$-valued semimartingale and $Y$ is an $\H$-valued semimartingale, then by Theorem \ref{SIdef2}, the stochastic integral $X_-\cdot Y$ exists. Now, by Lemma \ref{barmap}, $\bar{Y}$ is an $L(\LL,\K)$-valued process, and consequently, $\bar{Y}_-\cdot X$ exists. Define the (generalized) quadratic variation  process between $X$ and $Y$ as
\begin{align}\label{quadgen}
[[X,Y]]_t = X(t)(Y(t)) - X(0)(Y(0)) -X_-\cdot Y(t) -\bar{Y}_-\cdot X(t).
\end{align}
$[[X,Y]]$ is a $\K$-valued process and
$$[[X,Y]]_t = \lim_{\|\sigma\| \rt 0} \sum_i (X(t_{i+1}) - X(t_i))(Y(t_{i+1}) - Y(t_i)) $$
where $\sigma = \{t_i\}$ is a partition of $[0,t]$, and $\|\sigma\| = \sup(t_{i+1}-t_i)$ is the mesh of the partition $\sigma$.
The next result is a generalization of Lemma \ref{qd1}.

\begin{lemma}
\label{qd2}
Let $\H$ be  a separable Hilbert space and $Y$  an $\H$-valued adapted semimartingale. Let $\LL = \K\oth\H\equiv HS(\H,\K)$.
Suppose that $J$ and $ V$ are  cadlag, adapted  processes taking values in $HS(\K,\LL)$ and $HS(\H,\K)$ respectively.
Define  $X =  V_-\cdot Y$. Note that $X$ is a $\K$- valued semimartingale. 
Let  $U(t) = \int_0^t J(s-)dX(s)$. Then for any $\H$-valued adapted semimartingale $Z$, we have
\[[[U,Z]]_t = \int_0^t \tilde{J(s-) V(s-)} \ d[Z,Y]^\ot(s).\]
\end{lemma}

\begin{proof}
Similar to the previous lemma, we adopt the following notations:
Let $\s = \{t_i\}$ be a partition of $[0,t]$. For a process $H$, define $H^\sigma$
by
$$H^\s(s) = \sum_{i}H(t_i)1_{[t_i,t_{i+1}]}(s).$$
Let
\[X_\s(u) \equiv \int_0^u V^\s(s-) \ dY_s = \sum_i V(t_i\wedge u-) (Y(t_{i+1}\wedge u) - Y(t_i\wedge u)) \ \ \ \ 0\leq u \leq t\] and
\[U_\s(u) \equiv \int_0^u J^\s(s-)dX_\s(s) = \sum_i J(t_i\wedge u-) (X_\s(t_{i+1}\wedge u) - X_\s(t_i\wedge u)). \]
Denote
$$A = \int_0^t \tilde{J(s-)V(s-)} d[Z,Y]^\ot(s)$$
and notice that

\begin{align*} 
A&=  \lim_{\|\s\|\rt 0} \sum_i \tilde{J(t_i) V(t_i)}(Z(t_{i+1} - Z(t_i)))\ot (Y(t_{i+1}) - Y(t_i))   \\
& = \lim_{\|\s\|\rt 0} \sum_i \bar{(Z(t_{i+1} - Z(t_i))}J(t_i)V(t_i)(Y(t_{i+1}) - Y(t_i))  \\
& = \lim_{\|\s\|\rt 0} \sum_i \bar{(Z(t_{i+1} - Z(t_i))}(J(t_i)(X_s(t_{i+1}) - X_s(t_i)))\\
& = \lim_{\|\s\|\rt 0} \sum_i \bar{(Z(t_{i+1} - Z(t_i))}(U_\s(t_{i+1}) - U_\s(t_i))\\
& = \lim_{\|\s\|\rt 0} \sum_i (U_\s(t_{i+1}) - U_\s(t_i))(Z(t_{i+1}) - Z(t_i)) = [[U, Z]]_t.
\end{align*}
\end{proof}

Recall that for a function $\phi$ mapping $[0,\infty)$ to a Banach space,  the total variation $T_t(\phi)$ was defined in \eqref{FVdef}.
\begin{theorem}
\label{FUT1}
Let $\H$ be  a separable Hilbert space. Suppose that $Y_n =M_n + A_n$ is an adapted $\H$-valued semimartingale, where  $\{A_n\}$ is a sequence of $\H$-valued $\{\SC{F}^n_t\}$-adapted  processes  of locally finite variation and $\{M_n\}$ is a sequence of $\H$-valued $\{\SC{F}^n_t\}$-adapted local martingales . Then $\{Y_n\}$ is UT if for each $t>0$,  $\{T_t(A_n)\}$ is stochastically bounded (tight) and there exists a constant $C(t)$ such that $\EE([M_n,M_n]_t) <C(t)$. 
\end{theorem}

\begin{proof}  
It is enough to prove that $\{A_n\}$ and $\{M_n\}$ are UT.

For an $\{\SC{F}^n_t\}$-adapted, cadlag process $J$, we have 
\[|\int_0^tJ(s-)  \ dA_n(s)| \leq \int_0^t \|J(s-)\| \ dT_s(A_n). \]
Thus, if $\sup_{s\leq t}\|J(s)\|_{\H} \leq 1$, we have
\begin{align*}
P(\sup_{s\leq t} |\int_0^s J(r-) \ dA_n(r)| > K) \leq P(\int_0^t \|J(s-)\| \ dT_s(A_n) > K) \leq P(T_t(A_n) > K)
\end{align*}
which proves that $\{A_n\}$ is UT.

Next notice that
\begin{align*}
\PP (\sup_{s\leq t}|\int_0^s J(r-) dM_n(r)| > K)& \leq  \EE(\sup_{s\leq t}|\int_0^s J(r-) dM_n(r)| ^2)/K^2\\
& \leq 4 \EE(\int_0^{ t}\|J(r-)\|^2 d[M_n,M_n]_r)/K^2\\
& \leq \EE([M_n,M_n]_t)/K^2 \leq C(t)/K^2
\end{align*}
which proves that $\{M_n\}$ is UT.

\end{proof}


\section{Wong-Zakai type SDE}\label{WZmain}
We are now ready to state our main results. Notice that from Section \ref{tenSI}, the $H_n$ and $K_n$ defined below are $(\H\hat{\ot}_{HS}\H) = (\H\hat{\ot}_{HS}\H)^*$-valued semimartingales,  hence standard $(\H\hat{\ot}_{HS}\H)^\#$-semimartingales. In fact, by Theorem \ref{SIdef2}, for any Hilbert space $\X$ and an adapted cadlag process $\xi$ taking values in $L(\H\hat{\ot}_{HS}\H,\X)$, the stochastic integrals $\xi_-\cdot H_n$ and $\xi_-\cdot K_n$ exist. Therefore,  more generally the $H_n$ and $K_n$ are $(\X,\hat{\mathcal{H}})^\#$-semimartingales (see Remark \ref{lhsharp}), where $\hat{\mathcal{H}}$ can be taken to be the completion of the space $\X\ot(\H\hat{\ot}_{HS}\H)$ with respect to some norm such that $\hat{\mathcal{H}} \subset L(\H\hat{\ot}_{HS}\H,\X).$

\begin{theorem}
\label{sde_approx}
Let $\H$ be a separable Hilbert space. Let $Y_n, Z_n$ be two cadlag and adapted $\H$-valued semimartingales and  $f: \R \Rt \H $ a twice continuously differentiable  function with first and second-order derivatives denoted by $Df$ and $D^2f$ respectively. Define 
\[H_n(t)  = \int_0^t Z_n(s-)\ot dZ_n(s) \ ,  \  K_n(t) = [Y_n,Z_n]^\ot_t.\]
Suppose $X_n$ satisfies 
\begin{align}\label{sde_form}
X_n(t) = X_n(0) + \int_0^t f(X_n(s-))\ dY_n(s) + \int_0^t f(X_n(s-))\ dZ_n(s). 
\end{align}
Assume that $\{Y_n\}$ and $\{H_n\}$   are UT sequences,  and for each $t>0$, $\{[Z_n,Z_n]_t\}$ is a tight sequence. Also assume that there exist an $\H^\#$-semimartingale $Y$ and  $(\H\hat{\ot}_{HS}\H)^\#$-semimartingales $H,K$ such that
\[A_n:= (X_n(0), Y_n,Z_n,H_n, K_n) \RT (X(0),Y,0,H, K) :=A,\]
in the following sense: for any $\{h_i,h_i'\}_{i=1}^m \subset \H$ and $\{u_i,u'_i\}_{i=1}^m \subset \H\oth\H$
$$\{(X_n(0), Y_n(h_i,\cdot), Z_n(h'_i,\cdot), H_n(u_i,\cdot),K_n(u'_i,\cdot))\}_{i=1}^m \RT \{(X(0), Y(h_i,\cdot), 0, H(u_i,\cdot),K(u'_i,\cdot))\}_{i=1}^m $$
in $D_{\R\times \R^m\times \R^m\times\R^m\times \R^m}[0,\infty).$
Then $\{(A_n,X_n)\}$ is relatively compact, and any limit point $(A,X)$ satisfies 
\[X(t) = X(0) + \int_0^t f(X(s-))\ dY(s) + \int_0^t Df(X(s-))\ot f(X(s-))\ d(H^*(s) - K(s)). \]
\end{theorem}

\begin{remark} Notice that 
\[\int dZ_n(s)\ot Z_n(s-) = H_{n}^*,\]
where $H_{n}^*$ denotes adjoint of $H_n$.  Therefore, by the hypothesis
\[\int dZ_n(s)\ot Z_n(s-) \RT H^*.\]
\end{remark}

\begin{remark}
For a function $\phi$, let $\Delta \phi (s) = \phi(s) - \phi(s-)$. Notice that $\Delta H_n (s) = Z_n(s-) \ot \Delta Z_n(s) \RT 0$. It follows that $H$ is  continuous. Similarly, $K$ is continuous.
\end{remark}

\begin{proof}
By Remark \ref{tensorcov}, 
$$T_t([Z_n,Z_n]^\ot) \leq [Z_n,Z_n]_t.$$
Since for each $t>0$, $[Z_n,Z_n]_t$ is tight by the assumption, it follows from Theorem \ref{FUT1} that $\{[Z_n,Z_n]^\ot\}$ is UT.
Note that by the integration by parts formula for tensor stochastic integral, we have
\[[Z_n,Z_n]^{\ot}_t =  Z_n(t)\ot Z_n(t) - Z_n(0)\ot Z_n(0) - \int_0^t Z_n(s-)\ot dZ_n(s) - \int_0^t dZ_n(s)\ot Z_n(s-). \]
It follows from the hypothesis that
\[[Z_n,Z_n]^{\ot} \RT -(H + H^*).\]

\np
Observe that $T_t([Y_n,Z_n]^\ot) \leq [Y_n,Y_n]_t+ [Z_n,Z_n]_t$, and since $\{[Y_n,Y_n]_t\}$ and $\{[Z_n,Z_n]_t\} $ are tight for each $t>0$, it follows again from Theorem \ref{FUT1}  that $\{K_n\equiv [Y_n,Z_n]^\ot\}$ is UT.  \\

\np
By It{\^o}'s formula we have
\[f(X_n(t)) = f(X_n(0))+ \int_0^tDf(X_n(s-))\ dX_n(s) + R_n(t).\]
$$R_n(t) = \f{1}{2} \int_0^t D^2f(X_n(s-)) d[X_n,X_n]^{c}_s + \sum_{s\leq t } [\Delta f(X_n(\cdot))(s)  - Df(X_n(s-))\Delta X_n(s) ].$$
where $[X_n,X_n]^{c}_t = [X_n,X_n]_t - \sum_{s\leq t}\Delta X_n(s) \Delta X_n(s) $ is the continuous part of $[X_n,X_n]$.
It follows that  $\{R_n\}$ is a locally finite variation process. Notice that 
$T_t(R_n)$ is dominated by a linear combination of $[Z_n, Z_n]_t, [Y_n, Y_n]_t$, and since each of them is tight, we have $\{R_n\}$ to be UT. \\ 
Next, an application of the integration by parts formula (see (\ref{intpart1})) gives
\[\int_0^t f(X_n(s-))\ dZ_n(s) = \<f(X_n(s), Z_n(s)\> - \int_0^t Z_n(s-) \ df(X_n(s)) - [Z_n, f(X_n)]_t.\]
Now notice that
\[\int_0^t Z_n(s-) \ df(X_n(s)) = \int_0^t Z_n(s-) Df(X_n(s-)) \ dX_n(s) + \int_0^t Z_n(s-)dR_n(s).\]
Notice that $Z_n(s) \in \H^* = L(\H,\R)$ and $Df(X_n(s)) \in L(\R,\H)$. Therefore $Z_n(s-) Df(X_n(s-))$ is well defined and $\in L(\R,\R) \cong \R$.\\
Hence,
\begin{align*}
\int_0^t Z_n(s-) \ df(X_n(s))& = \int_0^t (Z_n(s-) Df(X_n(s-)))f(X_n(s-))dY_n(s) \\
& \hspace{.4cm}+ \int_0^t (Z_n(s-) Df(X_n(s-)))f(X_n(s-))dZ_n(s)+ \int_0^t Z_n(s-)dR_n(s)\\
& = \int_0^t (Z_n(s-) Df(X_n(s-)))f(X_n(s-))dY_n(s) \\
& \hspace{.4cm}+ \int_0^t Df(X_n(s-))\ot f(X_n(s-)) \ dH_n(s)+ \int_0^t Z_n(s-) \ dR_n(s),
\end{align*}
where $H_n(s) = \int_0^t Z_n(s-)\ot dZ_n(s)$, and the equality of the middle terms in the above two lines follows by Lemma \ref{ch1}.
Next by Lemma \ref{qd1}, we have
\begin{align*}
[Z_n, f(X_n)]_t& = \int_0^t Df(X_n(s-))\ot f(X_n(s-)) \ d[Z_n,Y_n]^\ot_s \\
& \hspace{.4cm}+ \int_0^t Df(X_n(s-))\ot f(X_n(s-)) \ d[Z_n,Z_n]^\ot_s + [Z_n,R_n]_t.
\end{align*}
Putting things together, we see that 
\[\int_0^t f(X_n(s-))\ dZ_n(s) = V_n(t) - \int_0^t Df(X_n(s-))\ot f(X_n(s-)) \ d(H_n(s) + K_n(s) + [Z_n, Z_n]^{\ot}(s),\]
where 
\begin{align*}
V_n(t) &= \<f(X_n(s)), Z_n(s)\> - \int_0^t (Z_n(s-)Df(X_n(s-)))f(X_n(s-))\ dY_n(s) \\
& \hspace{.4cm}- \int_0^tZ_n(s-) \ dR_n(s) - [Z_n,R_n]_t.
\end{align*}
From the hypothesis, we get $V^n \RT 0$. 
Plugging things back in the original equation, we have,
\begin{align*}
X_n(t)&=X_n(0)+V_n(t) + \int_0^t f(X_n(s-))\ dY_n(s) \\
& \hspace{.4cm} - \int_0^t Df(X_n(s-))\ot f(X_n(s-)) \ d(H_n(s) + K_n(s) + [Z_n, Z_n]^{\ot}(s)). 
\end{align*}
Take the indexing space $\Y = (\H, \H\hat{\ot}_{HS}\H)$ with the norm as $\|(h,u)\|_\Y = \|h\|_\H + \|u\|_{HS}.$
Consider $Y_n$ as an $\Y^\#$-semimartingale by defining
$$Y_n((h,u),\cdot) \equiv Y_n(h,\cdot).$$
Similarly, $H_n, K_n$ and $[Z_n, Z_n]$ can be considered as $\Y^\#$-semimartingales. 
Since $\{Y_n\},\{H_n\}, \{K_n\}$ and $\{[Z_n, Z_n]^\ot\}$ are UT, Theorem \ref{SDE_limit} gives the desired result.
 
\end{proof}

We next consider the case when the solutions of the stochastic differential equations of the form \eqref{sde_form} are also infinite-dimensional. We follows the steps in the above proof, however, the  difficulties lie in handling of infinite-dimensional It\^o's lemma, infinite-dimensional covariation, chain rule, appropriate integration by parts etc.  They are taken care of by suitable use of results from Section \ref{lemmas}. Notice that as in Theorem \ref{sde_approx}, the $H_n$ and $K_n$ defined below are $(\H\hat{\ot}_{HS}\H) = (\H\hat{\ot}_{HS}\H)^*$-valued semimartingales. 

\begin{theorem}
\label{sde_approx_inf}
Let $\H$ and $\K$ be  separable Hilbert spaces. Let $Y_n, Z_n$ be two cadlag and adapted $\H$-valued semimartingales and $f: \K \Rt \LL\equiv HS(\H,\K) $ a twice continuously differentiable  function with first and second-order derivatives denoted by $Df$ and $D^2f$ respectively. 
Notice that $Df :\K\rt L(\K,\LL)$. Assume that for each $x\in \K$, $D^2f(x)$ is an element of $L(\K\oth\K,\LL)$ and the mapping $x\rt D^2f(x)$ is  uniformly continuous on any bounded subset of $\K$.
 Define 
\[H_n(t)  = \int_0^t Z_n(s-)\ot dZ_n(s) \ ,  \  K_n(t) = [Y_n,Z_n]^\ot_t.\]
Suppose $X_n$ satisfies 
\begin{align}
\label{xn2}
X_n(t) = X_n(0) + \int_0^t f(X_n(s-))\ dY_n(s) + \int_0^t f(X_n(s-))\ dZ_n(s). 
\end{align}
Assume that $\{Y_n\}$ and $\{H_n\}$   are uniformly tight sequences,  and for each $t>0$, $\{[Z_n,Z_n]_t\}$ is a tight sequence.  Also assume that there exist an $\H^\#$-semimartingale $Y$ and  $(\H\hat{\ot}_{HS}\H)^\#$-semimartingales $H,K$ such that
\[A_n:= (X_n(0), Y_n,Z_n,H_n, K_n) \RT (X(0),Y,0,H, K) :=A,\]
in the same sense as in Theorem \ref{sde_approx}.
Then $(A_n,X_n)$ is relatively compact, and any limit point $(A,X)$ satisfies 
\[X(t) = X(0) + \int_0^t f(X(s-))\ dY(s) + \int_0^t \tilde{Df(X(s-)) f(X(s-))}\ d(H^*(s) - K(s)), \]
where the mapping \ \ $\tilde{}$\ \ was defined in Lemma \ref{tildemap}.
\end{theorem}

\begin{proof} As before, $\int_0^\cdot dZ_n(s)\ot Z_n(s-) \RT H^*, [Z_n,\cdot Z_n]^\ot \RT -(H+H^*)$ and  $\{[Z_n,Z_n]^\ot\}, \{K_n\equiv [Y_n,Z_n]^\ot\}$ are UT.

By the infinite-dimensional It{\^o}'s formula (Theorem \ref{infito}), we have
\begin{equation}
\label{ito2}
f(X_n(t)) = f(X_n(0))+ \int_0^tDf(X_n(s-))\ dX_n(s) + R_n(t).
\end{equation}
where $R_n$ is given by
$$R_n(t) = \f{1}{2} \int_0^t D^2f(X_n(s-)) d[X_n,X_n]^{c,\ot}_s + \sum_{s\leq t } [\Delta f(X_n(\cdot))(s)  - Df(X_n(s-))\Delta X_n(s) ].$$
where $[X_n,X_n]^{c,\ot}_t = [X_n,X_n]^\ot_t - \sum_{s\leq t}\Delta X_n(s)\ot \Delta X_n(s) $ is the continuous part of $[X_n,X_n]^\ot$.
As before, $\{R_n\}$ is  UT. \\ 
Next, an application of the integration by parts formula (see \eqref{quadgen}) gives
\[\int_0^t f(X_n(s-))\ dZ_n(s) = f(X_n(s) (Z_n(s)) - \int_0^t \bar{Z_n(s-)} \ df(X_n(s)) - [[f(X_n),Z_n]]_t,\]
where the mapping \ $\bar{\phantom{d}}$ \ is defined in the proof of Lemma \ref{barmap}.

Notice that by (\ref{ito2})
\[\int_0^t \bar{Z_n(s-)} \ df(X_n(s)) = \int_0^t \bar{Z_n(s-)} Df(X_n(s-)) \ dX_n(s) + \int_0^t \bar{Z_n(s-)}dR_n(s).\]

Hence, by (\ref{xn2}) 
\begin{align*}
\int_0^t \bar{Z_n(s-)} \ df(X_n(s))& = \int_0^t \bar{Z_n(s-) }Df(X_n(s-))f(X_n(s-))dY_n(s) \\
& \hspace{.4cm}+ \int_0^t \bar{Z_n(s-)} Df(X_n(s-))f(X_n(s-))dZ_n(s)+ \int_0^t \bar{Z_n(s-)}dR_n(s)\\
& = \int_0^t \bar{Z_n(s-)} Df(X_n(s-))f(X_n(s-))dY_n(s) \\
& \hspace{.4cm}+ \int_0^t \tilde{Df(X_n(s-)) f(X_n(s-))} \ dH_n(s)+ \int_0^t \bar{Z_n(s-)} \ dR_n(s),
\end{align*}
where $H_n(s) = \int_0^t Z_n(s-)\ot dZ_n(s)$, and the equality of the middle terms in the above two lines follows by  \eqref{eqch3}.
Next by Lemma \ref{qd2}

\begin{align*}
[[f(X_n),Z_n]]_t& = \int_0^t \tilde{Df(X_n(s-)) f(X_n(s-))} \ d[Z_n,Y_n]^\ot_s \\
& \hspace{.4cm}+ \int_0^t \tilde{Df(X_n(s-)) f(X_n(s-))} \ d[Z_n,Z_n]^\ot_s + [[R_n,Z_n]]_t.
\end{align*}

Putting things together, we see that 
\[\int_0^t f(X_n(s-))\ dZ_n(s) = V_n(t) - \int_0^t \tilde{Df(X_n(s-)) f(X_n(s-))} \ d(H_n(s) + K_n(s) + [Z_n, Z_n]^{\ot}(s),\]
where 
\begin{align*}
V_n(t) &= f(X_n(s)) (Z_n(s)) - \int_0^t \bar{Z_n(s-)}Df(X_n(s-))f(X_n(s-))\ dY_n(s) \\
& \hspace{.4cm}- \int_0^t\bar{Z_n(s-)} \ dR_n(s) - [[Z_n,R_n]]_t.
\end{align*}
From the hypothesis, we get $V^n \RT 0$. 
Plugging things back in the original equation, we have,
\begin{align*}
X_n(t)&=X_n(0)+V_n(t) + \int_0^t f(X_n(s-))\ dY_n(s) \\
& \hspace{.4cm} - \int_0^t \tilde{Df(X_n(s-)) f(X_n(s-))} \ d(H_n(s) + K_n(s) + [Z_n, Z_n]^{\ot}(s)) 
\end{align*}
Take the indexing space $\Y = (\H, \H\hat{\ot}_{HS}\H)$ with the norm as $\|(h,u)\|_\Y = \|h\|_\H + \|u\|_{HS}.$
Consider $Y_n$ as an $\Y^\#$-semimartingale by defining
$$Y_n((h,u),\cdot) \equiv Y_n(h,\cdot).$$
Similarly $H_n, K_n$ and $[Z_n, Z_n]$ can be considered as $\Y^\#$-semimartingales. 
Since $\{Y_n\},\{H_n\}, \{K_n\}$ and $\{[Z_n, Z_n]^\ot\}$ are UT,  the desired result now follows from Theorem \ref{SDE_limit2}.

\end{proof}

\begin{example}\label{NYex}{\rm
Let $U$ be an adapted semimartingale  taking values in a Hilbert space $\H$. Let $\{G_n\}$ be a sequence of adapted $\H$-valued semimartingales with $G_n\RT U$. Suppose that $G_n = M_n+A_n$ is a decomposition of the semimartingale $G_n$ into its local martingale and finite variation parts and  that $\{M_n\}$ and $\{A_n\}$ satisfy the assumptions of Theorem \ref{FUT1}. Note that this implies $\{G_n\}$ is UT.  In many examples $G_n\equiv U$. As a first example, consider the stochastic differential equation
\begin{align}\label{NYsde}
X_n(t) = X_n(0)+ \int _0^t\sigma(s,X_n(s), U_n(s)) dU_n(s),
\end{align}
where $\sigma: \R\times\R\times \H \Rt \H$ is twice continuously differentiable and
$$U_n(t) = G_n(\f{k}{n}) +n(t-\f{k}{n})\le(G_n(\f{k+1}{n}) -G_n(\f{k}{n})\ri),\quad \f{k}{n} \leq t < \f{k+1}{n}.$$
Notice that the $X_n$ are real-valued processes.
Let $\partial_i \s$ denote the partial derivative of $\s$ with respect to the $i$-th component. Notice that $\partial_1\s, \partial_2 \s\in \H$ and $\partial_3\s \in L(\H,\H).$ Assume that $\partial_3 \s \in HS(\H,\H)$.\\

\np
As discussed, $U, G_n$ and $U_n$ can be considered as $\H^\#$-semimartingales. It is easy to see that $\{U_n\}$ is not UT.
Let $U_n = Y_n + Z_n$, where $Y_n(t) = G_n(\f{[nt]+1}{n})$ and $Z_n = U_n - Y_n$. 
We claim that $\{Y_n\}$ is UT. To see this, write $Y_n(t) \equiv \bar{M}_n(t) +\bar{A}_n(t) \equiv M_n(\f{[nt]+1}{n}) + A_n(\f{[nt]+1}{n})$. Note that $\{\bar{M}_n\}$ is a sequence of martingales with respect to the filtration $\SC{F}^n_t \equiv\SC{F}_{[nt]+1}$, with $ E[\bar{M}_n,\bar{M}_n]_t \leq  E[M_n,M_n]_{t+1}$. Also, $T_t(\bar{A}_n) \leq T_{t+1}(A_n)$.  The assertion now follows by Theorem \ref{FUT1} and the assumptions on $\{M_n\}$ and $\{A_n\}$  .\\

\np
Next note that $\int_0^t Z_{n}(s-)\ot dZ_n(s) \rt - [U,U]^\ot_t/2$. 
To see this, note that since $Z_n(s-) = 0$, at the discontinuity points of $Y_n$
\begin{align*}
\int_0^t Z_{n}(s-)\ot dZ_n(s)& = \int_0^t Z_n(s-)\ot dU_n(s) - \sum_{s\leq t}Z_n(s-)\ot \Delta Y_n(s) =\int_0^t Z_n(s-)\ot dU_n(s)  \\
& = \sum_{k} n\int_{k/n}^{(k+1)/n} \le(n\le(s-k/n\ri)\le(G_n(\f{k+1}{n})-G_n(\f{k}{n})\ri)\right. \\
& \hspace{.5cm} +\le. G_n(\f{k}{n}) - G_n(\f{k+1}{n})\ri)\ot \le(G_n(\f{k+1}{n})-G_n(\f{k}{n})\ri) ds\\
& = n\sum_k \le(G_n\le((k+1)/n\ri) - G_n(k/n)\ri)\ot \le(G_n\le((k+1)/n\ri) - G_n(k/n)\ri)\\
&\hspace{.5cm}\int_{k/n}^{(k+1)/n} n(s-k/n) - 1 \ ds\\
& = -\f{1}{2}\sum_k \le(G_n\le((k+1)/n\ri) - G_n(k/n)\ri)\ot \le(G_n\le((k+1)/n\ri) - G_n(k/n)\ri) \\
&\RT -\f{[U,U]^\ot_t}{2}, \quad \mbox{since } \{G_n\} \mbox{ is  UT}.
\end{align*}
Also since, 
$$[Z_n,Z_n]^\ot = Z_n(t)\ot Z_n(t) -\int_0^t Z_n(s-)\ot dZ_n(s) -\int_0^t dZ_n(s-) \ot Z_n(s),$$
it follows that
$$[Y_n,Z_n]^\ot  = -[Y_n,Y_n]^\ot= -[Z_n,Z_n]^\ot \RT [U,U]^\ot$$
Moreover, $T_t(\int Z_{n-}\ot dZ_n) \leq n\sum_k \|\le(G_n\le((k+1)/n\ri) - G_n(k/n)\ri)\|^2 \int_{k/n}^{(k+1)/n}|n(s-k/n) - 1| \ ds \rt t \ [U,U]_t/2.$
It follows that for each $t>0$, $\{T_t(\int Z_{n-}\ot dZ_n)\}$ is tight, and hence the sequence $\{\int Z_{n-}\ot dZ_n\}$ is UT.\\

\np
We next derive the limiting stochastic differential equation for \eqref{NYsde}. Define 
\begin{align*}
\tilde{X}_n(t) = (t, X_n(t), U_n(t))^T,\ \tilde{U}_n(t) = (t, U_n(t), U_n(t))^T, \ \tilde{U}(t) = (t, U(t), U(t))^T
\end{align*}
and $F:\R\times \R\times \H \rt L(\R\times \H\times \H, \R\times\R\times \H)$ by
\begin{align*}
F(t,x,h) = \begin{pmatrix}
1& 0& 0\\
0&\sigma(t,x,h)&0\\
0&0&1
\end{pmatrix}
\end{align*}
In other words, the operator $F(t,x,h)$ is defined as
\begin{align}\label{opF}
F(t,x,h) y = (y_1, \<\sigma(t,x,h),y_2\>_\H, y_3)^T \in  \R\times\R\times \H, \quad y = (y_1,y_2,y_3)^T \in  \R\times \H\times \H.
\end{align}
Note that \eqref{NYsde} implies
\begin{align}
\label{KYsde2}
\tilde{X}_n(t) =\tilde{X}_n(0) + F(\tilde{X}_n) \cdot \tilde{U}_n(t) =\tilde{X}_n(0) + F(\tilde{X}_n) \cdot \tilde{Y}_n(t)+F(\tilde{X}_n) \cdot \tilde{Z}_n(t),
\end{align}
where $\tilde{Y}_n(t) = (t, Y_n(t), Y_n(t))$ and $\tilde{Z}_n(t) = (0, Z_n(t), Z_n(t)).$ Now the previous discussion tells that the sequences of $\R\times \H\times \H$-valued processes $\{\tilde{Y_n}\}$ and $HS(\R\times \H\times \H,\R\times \H\times \H)$-valued processes $\{\int \tilde{Z}_n \ot d\tilde{Z}_n \}$ are UT, and
\begin{align*}
\tilde{Y}_n \RT \tilde{U}, \int \tilde{Z}_n \ot d\tilde{Z}_n \RT -\f{[\tilde{U},\tilde{U}]^\ot_t}{2} \equiv -\f{1}{2}\begin{pmatrix}
0& 0& 0\\
0&[U,U]^\ot&[U,U]^\ot\\
0&[U,U]^\ot&[U,U]^\ot
\end{pmatrix} \in HS(\R\times \H\times \H,\R\times \H\times \H)
\end{align*}
and 
\begin{align*}
[\tilde{Y}_n,\tilde{Z}_n]^\ot  = -[\tilde{Y}_n,\tilde{Y}_n]^\ot= -[\tilde{Z}_n,\tilde{Z}_n]^\ot \RT -\begin{pmatrix}
0& 0& 0\\
0&[U,U]^\ot&[U,U]^\ot\\
0&[U,U]^\ot&[U,U]^\ot
\end{pmatrix}.
\end{align*}
Here, the elements of $ HS(\R\times \H\times \H,\R\times \H\times \H)$ are represented by matrix like structures of the form:
\begin{align*}
\Xi = \begin{pmatrix}
\beta&h_{12}&h_{13}\\
h_{21}&\xi_{22}&\xi_{23}\\
h_{31}&\xi_{32}&\xi_{33}
\end{pmatrix}, \quad \beta \in \R, \ h_{ij} \in \H,\  \xi_{ij} \in \H\oth\H.
\end{align*}
In other words, the Hilbert-Schmidt operator $\Xi \in  HS(\R\times \H\times \H,\R\times \H\times \H)$  is defined as
$$\Xi(y) = (\beta y_1+ \<h_{12},y_2\>_\H+\<h_{13},y_3\>_\H, y_1h_{21}+ \xi_{22}(y_2)+\xi_{23}(y_3),y_1h_{31}+ \xi_{32}(y_2)+\xi_{33}(y_3) )$$
Observe that the derivative operator, $DF(t,x,h) \in L(\R\times \R\times \H, L(\R\times \H\times \H, \R\times\R\times \H))$ is given by
\begin{align}\label{opDF}
DF(t,x,h)(b) = \begin{pmatrix}
0&0&0\\
0&b_1\partial_1\s+b_2\partial_2 \s+ \partial_3\s b_3&0\\
0&0&0
\end{pmatrix}, \quad b \in \R\times \R\times \H. 
\end{align}

\np
Now an application of Theorem \ref{sde_approx_inf} to \eqref{KYsde2} gives the limiting stochastic differential equation as
\begin{align}\label{limsde}
\tilde{X}(t) = \tilde{X}(0) + \int_0^t F(\tilde{X}(s-))\ d\tilde{U}(s) + \f{1}{2}\int_0^t \tilde{DF(\tilde{X}(s-)) F(\tilde{X}(s-))}\ d[\tilde{U},\tilde{U}]^\ot_t,
\end{align}
where  the mapping \ \ $\tilde{}$\ \ was defined in Lemma \ref{tildemap}.
Observe that in the present example (see Section \ref{prf} for a proof),
\begin{align}\label{tildefunct}
\tilde{DF(\tilde{x})F(\tilde{x})}(\Xi) = (0, \<\partial_1 \s, h_{21}\>_\H+\<\partial_2\s\ot \s,\xi_{22}\>_{\H\oth\H}+\<\partial_3\s,\xi_{23}\>_{\H\oth\H}, 0)^T.
\end{align}
Therefore, considering the middle component of \eqref{limsde}, it follows that $X_n \RT X$ where $X$ satisfies
$$X(t) = X(0) + \int_0^t \s dU(s)+ \f{1}{2}\int_0^t( \partial_3 \s+ \partial_2 \s\ot \s )d[U,U]^\ot_s$$

}
\end{example}

\begin{remark}
Example \ref{NYex} is a generalization of the results obtained by Nakao and Yamato \cite{NY78} (see Theorem \ref{WZ03}) and Konecny \cite{Konecny83} to stochastic differential equations driven by infinite-dimensional semimartingales. One important example of $U$ in Example \ref{NYex} is an $\H$-valued Brownian motion $W$ with covariance operator $Q$, where $Q$ is nuclear (see \cite{DPZ92}). In other words, for $h_1,h_2 \in \H$,
$[W(h_1, \cdot), W(h_2, \cdot)]_t = \<Qh_1,h_2\>_\H t$. The tensor quadratic variation of the process $W$ is given by $[W,W]^\ot_t = tQ$.
\end{remark}

\begin{remark}
Using the same technique, Example \ref{NYex} can easily be extended to the case where the solutions $X_n$ are also infinite-dimensional. Also, the approximation of the semimartingale $U$ by linear interpolation is just chosen for illustrative purpose. It can be easily extended to more general approximation techniques.
\end{remark}

\def\hsop{S}

\begin{example}
\label{Gauss_conv}
{\rm
 As a second example, we consider a space-time Gaussian white noise and its mollified version as its approximation. More precisely,  let $W$ be an $\{\SC{F}_t\}$-adapted space-time Gaussian white noise and $B_r(x)\subset \R^d$ denote the ball of radius $r$, centered at $x$. 
Let $\rho: \R^d \rt [0,\infty)$ and $\eta:\R \rt [0,\infty)$ be  smooth functions with $supp(\rho) \subset B_1(0)$, $supp(\eta) \subset (-1,0)$, and $\int_{\R^d} \rho(x) \ dx= 1, \int_{-\infty}^\infty \eta(s) \ ds =1$.\\
Define $\rho_n(x) = n^d\eta(nx)$, and $\eta_n(s) = n\eta(nx)$. Notice that 
 $\rho_n$ is supported on $B_{1/n}(0) \subset \R^d$, and $\eta_n$ is supported on $[-1/n, 0]$, and $\int_{B_{1/n}(0)}\rho_n(x) \ dx = 1$,  $\int_{[-1/n, 0]} \eta_n(s) = 1$.\\
 Define 
\begin{align}
 \label{Wapprox0}
\dot{W_n}(x,t) = \int_{\R^d\times [0,\infty)}\rho_n(x-y)\eta_n(s-t) W(dy\times ds).
\end{align}
For $h\in L^2(\R^d)$, let 
\begin{align}
 \label{Wapprox1}
{W_n}(h,t) = \int_{\R^d\times [0,\infty)} h(x) \dot{W_n}(x,s) dx \ ds,
\end{align}
Notice that $\{W_n\}$ is a sequence of $\{\SC{F}_t\}$-adapted $\H^\#$-semimartingales, for $\H = L^2(\R^d)$, and  $W_n \stackrel{P}\rt W$ in the sense that, for any finite $h_1,\hdots,h_m \in L^2(\R^d)$
$$(W_n(h_1,\cdot),\hdots, W_n(h_m,\cdot)) \stackrel{P}\rt (W(h_1,\cdot),\hdots, W_n(h,\cdot)).$$
Consider the following SDE
$$X_n(t) = X_n(0) + \int_{\R^d\times [0,t)}g(X_n(s), x) \dot{W_n}(x,s) dx \ ds.$$
Assume that
\begin{itemize}
 \item $|g(\cdot,x)| \leq \kappa(x)$ for some $\kappa \in L^1$ so that the integration in the right side is defined;

\item $g(y,x) = \hsop f(y,\cdot)(x)$, where $\hsop $ is a Hilbert-Schmidt operator on $L^2(\R^d)$. For example, if $\int_{\R^d\times \R^d}\gamma^2(x,u) dx \ du < \infty$, then $\hsop$ could be defined as
\begin{align}
 \label{Pex}
\hsop h(x) = \int_{\R^d}h(u)\gamma(x,u) \ du;
\end{align}
(In other words, if $S$ is defined by \eqref{Pex}, then $g(y,x) = \int_{\R^d}f(y,u)\gamma(x,u) \ du.$)

\item $\sup_{u}\int_{R^d}|f(u,x)|^2 \ dx < \infty, \sup_{u}\int_{R^d}|\partial_1f(u,x)|^2 \ dx   <\infty$ and $\sup_{u}\int_{R^d}|\partial_1^2f(u,x)|^2 \ dx   <\infty$, where $\partial_1f$ and $\partial_1^2f$ denote the first and second order partial derivative of $f$ with respect to the first co-ordinate.
\end{itemize}
The above assumptions imply that the mapping
$$u\in \R^d \rt f(u,\cdot)\in L^2(\R^d)$$ 
is  bounded in $L^2(\R^d)$ with bounded first and second-order  (Frechet) derivative.  \\
\np
Thus  $X_n$ satisfies
\begin{align}
\label{whitesde_approx}
X_n(t) = X_n(0) + \hsop f(X_n(\cdot), \cdot) \cdot W_n(t),
\end{align}
that is 
$$X_n(t) = X_n(0) + \int_{\R^d\times [0,t)}\hsop f(X_n(s), \cdot)(x) \dot{W_n}(x,s) dx \ ds.$$
Observe that $\{W_n\}$ is not a UT sequence, as
$$\int_0^tW_n(h,s)dW_n(h,s) = \f{1}{2}W_n(h,t)^2 \nRightarrow \int_0^t W(h,s) dW(h,s) = \f{1}{2}(W(h,t)^2 - \|h\|^2t).$$
We apply Theorem \ref{sde_approx} to find the limit of $\{X_n\}$. \\
First, notice that the SDE (\ref{whitesde_approx}) could be written as
\begin{align}
\label{whitesde_approx2}
X_n(t) =X_n(0) + f(X_n(\cdot), \cdot) \cdot \SC{W}_n(t),
\end{align}
where $\SC{W}_n$ is defined as
\begin{align}
 \label{Wapprox2}
{\SC{W}_n}(h,t) = \int_{\R^d\times [0,t)} (\hsop h)(x) \dot{W_n}(x,s) dx \ ds,
\end{align}
It is easy to check that 
$$E[\sup_{s \leq t}|\SC{W}_n(h,s) - W(\hsop h,s)|^2] \rt 0.$$
Observe that
\begin{align}
\non
{\SC{W}_n}(h,t)&  = \int_{\R^d\times [0,t)} (\hsop h)(x)(\int_{\R^d\times [0,\infty)}\rho_n(x-y)\eta_n(r-s)W(dy\times dr) ) dx \ ds\\
\non
& =\int_{\R^d\times [0,t)} (\int_{\R^d\times [0,t)} (\hsop h)(x) \rho_n(x-y)\eta_n(r-s) dx \ ds) W(dy\times dr) \\
& = \int_{\R^d\times [0,t)}(S_n \hsop h)(y) (\int_0^t ( \eta_n(r-s)  \ ds) W(dy\times dr) 
 \label{hwn2}
\end{align}
where the operator $S_n$ is defined as 
$$S_nh(x) = \int_{\R^d}h(y)\rho_n(x-y) \  dy.$$
Note that $\|S_n\|_{op} \leq 1.$
Write 
$$\SC{W}_n(h,t) = Y_n(h,t) + Z_n(h,t).$$
where $Y_n(h,t) \equiv W(S_n \hsop h,t).$
Define 
$$\tilde{\SC{W}}_n(t) = \sum_j \SC{W}_n(e_j,t)e_j$$
Notice that the infinite sum above converges, as from (\ref{hwn2})
\begin{align*}
\sup_{s\leq t}E(\|\sum_{j=K}^M\SC{W}_n(e_j,t)e_j\|_2^2 &= \sup_{s\leq t}\sum_{j=K}^M E(\SC{W}_n(e_j,s)^2) \\
& \leq \sum_{j=K}^M \|S_n \hsop e_j\|^2_2\ t \\
& \leq \sum_{j=K}^M \| \hsop e_j\|^2_2\ t, \ \ \ \mbox{as } \|S_n\|_{op} \leq 1 \\
& \rt 0, \mbox{ as } K,M \rt \infty , \  \mbox{ since }  \hsop \mbox{ is Hilbert-Schmidt.} 
\end{align*}

\np
It follows that $\tilde{\SC{W}}_n\in D_{L^2(\R^d)}[0,\infty)$. Similarly, 
$$\tilde{Y}_n \equiv \sum_j Y_n(e_j,\cdot) e_j \in D_{L^2(\R^d)}[0,\infty).$$
Thus, $\tilde{\SC{W}}_n$ and $\tilde{Y}_n$ are versions of $\SC{W}_n$ and $Y_n$ taking values in $L^2(\R^d)$ and it is easily checked that
$H\cdot \SC{W}_n = H\cdot \tilde{\SC{W}}_n$ and $H\cdot Y_n$ = $H\cdot \tilde{Y}_n.$
With a slight abuse of notation, we will continue to use $Y_n$ and $\SC{W}_n$ instead of $\tilde{Y}_n$ and $\tilde{\SC{W}}_n$, and consider them as $L^2(\R^d)$-valued semimartingales.\\

Put $K^n \equiv [Y_n,Z_n]^\ot = -[Z_n,Z_n]^\ot = - [Y_n, Y_n]^\ot$, and notice that
\begin{align*}
 [Y_n, Y_n]^\ot_t &= \sum_{j,k}[Y_n(e_j,\cdot), Y_n(e_j,\cdot)]_te_j\ot e_k\\
& =t\sum_{j,k} \<S_n\hsop e_j,S_n \hsop e_k\>e_j\ot e_k = t(S_n\hsop)^*(S_n\hsop).
\end{align*}
Here $\<\cdot,\cdot\>$ denotes the inner product in $L^2(\R^d)$. 
Recall that  $K_n$ is an $\H\oth\H$-valued (hence a standard ($\H\oth \H)^\#$ ) semimartingale.
We verify $K_n(t) \equiv [Y_n,Z_n]^\ot_t = -t(S_n\hsop)^*(S_n\hsop)$ converges to $-t\hsop^*\hsop$ in the sense of convergence of $(\H\oth \H)^\#$-semimartingale, that is we need to verify
for $u = \sum_{i=1}^{I}x_i\ot y_i$, $\<u, tK_n\> \rt \<u, tI\>$. This follows because
\begin{align*}
 \<u, tK_n\>& = \sum_{i}\<x_i, tK_ny_i\> =
 -t \sum_{i}\<x_i, (S_n\hsop)^*(S_n\hsop)y_i\> \\
&=- t \sum_{i}\<S_n\hsop x_i, S_n\hsop y_i\>  \Rt -t \sum_{i}\<\hsop x_i, \hsop y_i\>.
\end{align*}
The last equality is because $S_n\hsop  \rt \hsop$ in the strong operator topology.\\
Similarly, $H_n(t) \equiv \int_0^t Z_n(s)\ot dZ_n(s) \rt -\f{t}{2}\hsop^*\hsop.$

Next, observe that $\{Y_n\}$ is a uniformly tight sequence, and  $Y_n  \stackrel{P}\rt \SC{W}$ in the sense that, for any finite $h_1,\hdots,h_m \in L^2(\R^d)$
$$(Y_n(h_1,\cdot),\hdots, Y_n(h_m,\cdot)) \stackrel{P}\rt (\SC{W}(h_1,\cdot),\hdots, \SC{W}(h,\cdot)),$$
where $\SC{W}(h,t)  =W(\hsop h,t)$, that is, $\SC{W}$  is a space-time Gaussian white noise with 
$$[\SC{W}(h,\cdot),\SC{W}(g,\cdot)]_t = t\<\hsop h,\hsop g\>.$$
To apply Theorem \ref{sde_approx}, we only need to prove that the sequence $\{H_n \equiv \int_0^t Z_n(s-)\ot dZ_n(s) \}$ is uniformly tight.
For this purpose, we first compute $E\|Z_n(t)\|^2_2$.
Notice that 
$\|Z_n(t)\|^2_2  = \sum_{j}\|Y_n(e_j,t)-\SC{W}_n(e_j,t)\|^2$.
For any $h \in L^2(\R^d)$, we have using (\ref{hwn2})

\begin{align*}
 \SC{W}_n(h,t) - Y_n(h,t) &=  \int_{\R^d\times [0,t)}(S_n\hsop h)(y) (\int_0^t ( \eta_n(r-s)  \ ds) W(dy\times dr) \\
&\hspace{.4cm} -  \int_{\R^d\times [0,t)}(S_n\hsop h)(y) W(dy\times dr) \\
& = \int_{\R^d\times [0,t)}(S_n\hsop h)(y) (\int_0^t  \eta_n(r-s)  \ ds - 1) W(dy\times dr)  \\
&  =  \int_{\R^d\times [0,t-1/n)}(S_n\hsop h)(y) (\int_0^t  \eta_n(r-s)  \ ds - 1) W(dy\times dr)  \\
& \hspace{.4cm}+  \int_{\R^d\times [t-1/n,t)}(S_n\hsop h)(y) (\int_0^t  \eta_n(r-s)  \ ds - 1) W(dy\times dr)\\
& = \int_{\R^d\times [t-1/n,t)}(S_n\hsop h)(y) (-\int_t^{r+1/n} \eta_n(r-s)  \ ds ) W(dy\times dr).
\end{align*}
The last equality is because
$ \int_0^t  \eta_n(r-s)\ ds  =1, \mbox{ if }\ t\geq r+1/n$. Thus,
\begin{align*}
 E(\SC{W}_n(h,t) - Y_n(h,t) )^2 =\int_{\R^d\times [t-1/n,t)}|S_n\hsop h(y)|^2 dy \ dr \leq \f{1}{n}\|S_n\hsop h\|_2^2
\end{align*}
It follows that
\begin{align}
\label{Znorm}
 E\|Z_n(t)\|^2_2 \leq \f{1}{n}\sum_j \|S_n\hsop e_j\|_2^2 =\f{1}{n}\|S_n\hsop \|_{HS} \leq \f{1}{n}\|\hsop \|_{HS}.
\end{align}
Take $G$ to be an  $\{\SC{F}^n_t\}$-adapted $(\H\oth\H)$-valued cadlag process, and  $\|G(s)\|_{HS} \leq 1.$
Notice that by Theorem \ref{ch1} 
\begin{align*}
 G_n\cdot H_n(t)& = \int_0^t G(s-)(Z_n(s-)) \ dZ_n(s)  \\
& = \int_0^t G(s-)(Z_n(s-)) \ dY_n(s) - \int_0^t G(s-)(Z_n(s-)) \ d\SC{W}_n(s)\\
& = A + B.
\end{align*}
We have,
\begin{align*}
 E(B^2)& = \int_0^t\|S_n\hsop [G(s) (Z_n(s))]\|^2_2 \ ds \\
& \leq \int_0^t \|S_n\hsop \|_{op}\|G(s)\|_{op}\|Z_n(s)\|_2^2 \ ds \leq \int_0^t \|S_n\hsop \|_{op}\|G(s)\|_{HS}\|Z_n(s)\|_2^2 \ ds\\
&\leq \|\hsop\|_{op} \int_0^t\|Z_n(s)\|^2\ ds \leq \f{1}{n}\|\hsop\|_{HS}\|\hsop\|_{op} t
\end{align*}
and
\begin{align*}
A &= \int_0^t G(s-)(Z_n(s-))d\SC{W}_n(s)\\
& = \int_0^t \hsop G(s-)(Z_n(s-))(x)\int_{\R^d\times [0,\infty)}\rho_n(x-y)\eta_n(r-s)W(dy\times dr) dx ds\\
& = \int_0^t \int_{\R^d\times [0,\infty)}S_n\hsop G(s-)(Z_n(s-))(y)\eta_n(r-s)W(dy\times dr) ds.
\end{align*}
Thus,
\begin{align*}
 E(A^2) &\leq  \int_0^t E[\int_{\R^d\times [0,\infty)}S_n\hsop G(s-)(Z_n(s-))(y)\eta_n(r-s)W(dy\times dr)]^2 ds\\
& = \int_0^t( \int_{\R^d\times [0,\infty)}E|S_n\hsop G(s-)(Z_n(s-))(y)|^2|\eta_n(r-s)|^2 \ dy\ dr) \ ds\\
& \leq Cn\int_0^t( \int_{\R^d}E|S_n\hsop G(s-)(Z_n(s-))(y)|^2 dy) \ ds, \ \ \ C =\int_{-\infty}^{\infty}\eta^2(r) \ dr\\
& = Cn\int_0^tE\|S_n\hsop G(s-)Z_n(s-)\|_2^2  \ ds\\
& \leq Cn\|\hsop\|_{op} \int_0^tE \|Z_n(s-)\|^2_2\  ds\\
& \leq Ct \|\hsop\|_{op}\|\hsop\|_{HS}, \ \  \mbox{ using} \ (\ref{Znorm}).
\end{align*}
It follows that $\{H_n\}$ is uniformly tight.

\np
Now if $X_n(0) \stackrel{P}\Rt X(0)$, then applying Theorem \ref{sde_approx}, we conclude $X_n  \stackrel{P}\rt X$, where $X$ satisfies
\begin{align*}
X(t)& = X(0) + \int_{\R^d\times [0,t)} \hsop f(X(s),\cdot)(x)W(dx\times ds) \\
&\hspace{.4cm}+\int_0^t Df(X_{s-})\ot f(X_{s-})\ d(-\f{s}{2}\hsop\hsop^* +s \hsop^*\hsop).
\end{align*}
If $\hsop = \hsop^*$, which will be the case if $\hsop$ is defined by (\ref{Pex}), then the above SDE could be written as
\begin{align*}
X(t)& = X(0) + \int_{\R^d\times [0,t)} \hsop f(X(s),\cdot)(x)W(dx\times ds) \\
& \hspace{.4cm} +\f{1}{2} \int_{\R^d \times [0,t) }  \hsop f(X(s),\cdot)(x) \partial_1\hsop f(X(s),\cdot)(x) \ dx  \ ds.
\end{align*}

}
\end{example}

\begin{example}
{\rm
Let $\{\xi_n\}$ be a $\phi$-irreducible Markov chain taking values in a separable metric space $U$. Let $P$ denote the transition kernel of $\{\xi_n\}$. Assume that the chain is ergodic with unique stationary distribution $\pi$. Let $\H = L^2(U,\pi)$.
Let $\{\tilde{W}_n\}, \{\tilde{Y}_n\} $ and $\{\tilde{Z}_n\}$ be $\H^\#$-semimartingales  defined by
\begin{align*}
\tilde{W}_n(h,t)& \equiv \f{1}{\sqrt{n}}\sum_{k=1}^{[nt]}\le(Ph(\xi_k) - h(\xi_k)\ri)\\
& = \f{1}{\sqrt{n}}\sum_{k=1}^{[nt]} \le(Ph(\xi_{k-1}) - h(\xi_k)\ri)+ \f{1}{\sqrt{n}}\le(Ph(\xi_{[nt]}) - h(\xi_0)\ri)\\
& \equiv \tilde{Y}_n(h,t) + \tilde{Z}_n(h,t).
\end{align*}
Let $\hsop$ be a Hilbert-Schmidt operator from $\H$ to $\H$. Define 
$Y_n(h,t) \equiv \tilde{Y}_n(\hsop h,t)$ and $Z_n(h,t) \equiv \tilde{Z}_n(\hsop h,t)$.
 Let $\{e_k\}$ be an orthonormal basis of $\H$. With a slight abuse of notation, define
 \begin{align*}
Y_n(t) = \sum_k Y_n(e_k,t)e_k, \ \ Z_n(t) = \sum_k Z_n(e_k,t)e_k
\end{align*}
Then $Y_n$ and $Z_n$ are $\H$-valued processes. To see this, first note that
\begin{align}
\label{eq_markov1}
\sup_{t\leq T}\EE[\|\sum_{j=K}^M Y_n(e_j,t)e_j\|_2^2]& = \sup_{t\leq T}\sum_{j=K}^M\EE[\| Y_n(e_j,t)\|_2^2].
\end{align}
Now observe that for any $h\in \H$, 
\begin{align*}
\EE[\| Y_n(h,t)\|_2^2]& = \f{1}{n}\sum_{j=1}^{[nt]}\EE[P\hsop h(\xi_{k-1}) -\hsop h(\xi_k)]^2 \leq 2\f{[nt]}{n}(\|P\hsop h\|_2^2 + \|\hsop h\|_2^2) \leq 4\f{[nt]}{n}\|\hsop h\|_2^2.
\end{align*}
It follows from \eqref{eq_markov1} that 
\begin{align*}
\sup_{t\leq T}\EE[\|\sum_{j=K}^M Y_n(e_j,t)e_j\|_2^2]& \leq 4\f{[nT]}{n} \sum_{j=K}^M \|Se_j\|_2^2\\
& \rt 0, \mbox{ as } K,M\rt \infty, \mbox{ since } \hsop \mbox{ is Hilbert-Schmidt}.
\end{align*}
Similarly, $Z_n$ is an $\H$-valued process. \\
Consider a sequence of SDEs of the form \eqref{xn2} driven by $\{Y_n\}$ and $\{Z_n\}$. We show that $\{Y_n\}$ and $\{Z_n\}$ satisfy the assumptions of Theorem \ref{sde_approx_inf}.\\

\np
For each $n$, $Y_n$ is a martingale, and by the martingale central limit theorem it follows that for any collection of $h_1,\hdots,h_m \in \H$
$$(Y_n(h_1,\cdot),\hdots,Y_n(h_m\cdot)) \RT W,$$
where $W$ is an $m$-dimensional Gaussian process with covariance matrix, $tC$, and $C$ is given by
\begin{align*}
C_{i,j} & =\lim_{n\rt \infty}\f{1}{n}\sum_{k=1}^n (P\hsop h_i(\xi_{k-1})  -\hsop h_i(\xi_{k}))(P\hsop h_j(\xi_{k-1})  -\hsop h_j(\xi_{k}))\\
& = \int \pi(dx) \int P(x,dy)(P\hsop h_i(x) - \hsop h_i(y))(P\hsop h_j(x) - \hsop h_j(y)).
\end{align*}
Next, we prove that $\{Y_n\}$ is UT. By Theorem \ref{FUT1},  it is enough to show  that $\sup_n \EE[Y_n,Y_n]_t <\infty.$
To see this,   observe that
\begin{align*}
[Y_n,Y_n]_t = \mbox{trace}([Y_n,Y_n]^\ot_t) = \sum_k [Y_n(e_k,\cdot),Y_n(e_k,\cdot)]_t,
\end{align*}
where  $\{e_k\}$ is an orthonormal basis of $\H$.\\
Note that $[Y_n(h,\cdot), Y_n(g,\cdot)]_t =\f{1}{n}\sum_{k=1}^{[nt]} (P\hsop h(\xi_{k-1})  -\hsop h(\xi_{k}))(P\hsop g(\xi_{k-1})  -\hsop g(\xi_{k}))$ and therefore,
\begin{align*}
\EE[Y_n,Y_n]_t& = \f{[nt]}{n}\sum_k \int \pi(dx) \int P(x,dy)(P\hsop e_k(x) - \hsop e_k(y))^2\\
& \leq 4\f{[nt]}{n} \sum_k \|\hsop e_k\|^2_2.
\end{align*}
Since $\hsop$ is Hilbert-Schmidt, it follows that  $\sup_n \EE[Y_n,Y_n]_t <\infty.$

Also, it is immediate that $Z_n \RT 0$, in the sense that for any collection of $h_1,\hdots,h_m \in \H$
$(Z_n(h_1,\cdot),\hdots,Z_n(h_m\cdot)) \RT 0.$
Next, note that since $Z_n(h,t) = \f{1}{\sqrt{n}} \sum_{k=1}^{[nt]} (P\hsop h(\xi_k) - P\hsop h(\xi_{k-1})$
\begin{align*}
[Z_n(g,\cdot), Z_n(h,\cdot)]_t &= \f{1}{n} \sum_{k=1}^{[nt]} (P\hsop g(\xi_k) - P\hsop g(\xi_{k-1})(P\hsop h(\xi_k) - P\hsop h(\xi_{k-1})\\
& \RT t\int \pi(dx) \int P(x,dy) (P\hsop g(y) - P\hsop g(x))(P\hsop h(y) - P\hsop h(x)).
\end{align*}
It follows that for any $g_1,h_1, \hdots,g_m,h_m \in \H$,
\begin{align*}
([Z_n,Z_n]^\ot_t(g_1\ot h_1),\hdots, [Z_n,Z_n]^\ot_t(g_m\ot h_m))& =([Z_n(g_1,\cdot), Z_n(h_1,\cdot)]_t,\hdots,[Z_n(g_m,\cdot), Z_n(h_m,\cdot)]_t)\\
& \RT t\rho,
\end{align*}
where $\rho = (\rho_i)_{i=1}^m$ and $\rho_i =\int \pi(dx) \int P(x,dy) (P\hsop g_i(y) - P\hsop g_i(x))(P\hsop h_i(y) - P\hsop h_i(x)).$\\
\np
Also,
\begin{align*}
[Z_n(g,\cdot), Y_n(h,\cdot)]_t &= \f{1}{n} \sum_{k=1}^{[nt]} (P\hsop g(\xi_k) - P\hsop g(\xi_{k-1})(P\hsop h(\xi_{k-1}) - \hsop h(\xi_{k})\\
& \RT t\int \pi(dx) \int P(x,dy) (P\hsop g(y) - P\hsop g(x))(P\hsop h(x) - \hsop h(y)).
\end{align*}
Therefore,
\begin{align*}
([Z_n,Y_n]^\ot_t(g_1\ot h_1),\hdots, [Z_n,Y_n]^\ot_t(g_m\ot h_m)) & \RT t\rho',
\end{align*}
where $\rho' = (\rho'_i)_{i=1}^m$ and $\rho'_i =\int \pi(dx) \int P(x,dy) (P\hsop g_i(y) - P\hsop g_i(x))(P\hsop h_i(x) - \hsop h_i(y)).$\\
Similarly,
$$(\int_0^t Z_{n}(s-)\ot dZ_n(s)(g_1\ot h_1),\hdots, \int_0^t Z_{n}(s-)\ot dZ_n(s)^\ot(g_m\ot h_m))  \RT t\rho'',$$
where  $\rho'' = (\rho''_i)_{i=1}^m$ and $\rho''_i =\int \pi(dx) \int P(x,dy) P\hsop g_i(x) (P\hsop h_i(y) - \hsop h_i(x)).$
\np
Finally, we need to prove that $\{\int Z_{n}(s-)\ot dZ_n(s)\}$ is UT. By Theorem \ref{FUT1}, it is enough to show that for every $t>0$,  $\{T_t(\int Z_{n}(s-)\ot dZ_n(s))\} $ is tight. Call $H_n =\int Z_{n}(s-)\ot dZ_n(s)$. Recall that if $\{e_k\}$ is an orthonormal basis of $\H$, then  $\{e_k\ot e_l\}$ forms an orthonormal basis of $\H\oth\H$. Hence, 
\begin{align*}
T_t(\int Z_{n}(s-)\ot dZ_n(s)) &= \sup_{\{t_i\}}\sum_i \sqrt{\sum_k|H_n(e_k\ot e_l,t_i)-H_n(e_k\ot e_l,t_{i-1})|^2}\\
& = \f{1}{n}\sum_{j=1}^{[nt]}\sqrt{\sum_{k,l}|P\hsop e_k(\xi_{j-1})|^2|P\hsop e_l(\xi_j) - P\hsop e_l(\xi_{j-1})|^2}\\
&\leq \sqrt{\sum_{k}\f{1}{n}\sum_{j=1}^{[nt]} |P\hsop e_k(\xi_{j-1})|^2}\sqrt{\sum_{l}\f{1}{n}\sum_{j=1}^{[nt]} |P\hsop e_l(\xi_j) - P\hsop e_l(\xi_{j-1})|^2}\\
& \RT \sqrt{\sum_k \|P\hsop e_k\|^2}\sqrt{\sum_l\int \pi(dx)\int P(x,dy) |P\hsop e_l(x) -P\hsop e_l(y)|^2} <\infty.
\end{align*}
It follows that $\{T_t(\int Z_{n}(s-)\ot dZ_n(s))\} $ is tight.

}
\end{example}

\setcounter{section}{0}
\setcounter{theorem}{0}
\setcounter{equation}{0}
\renewcommand{\theequation}{\thesection.\arabic{equation}}

\appendix
\section*{Appendix}
\renewcommand{\thesection}{A} 

\subsection{Tensor product}
\label{tenprod}
All the results in this section are from Ryan \cite{Rr02}.
Let $\X$, $\Y$ be two Banach spaces. Let $B(\X\times\Y, \Z)$ be the space of all bounded bilinear forms from $\X\times \Y \rt \Z$, that is set of all bilinear forms $A$ such that
$$\|A(x,y)\|_{\Z} \leq \gamma \|x\|_{\X}\|y\|_{\Y}, \  \mbox{for some } \ \gamma>0.$$
The smallest such constant $\gamma$ is the norm of $A$, and will be denoted by $\|A\|$. 
If $\Z =\R$, then we will denote $B(\X\times\Y, \Z)$ by $B(\X\times\Y)$.

For a vector space $V$, let $V^\#$ denote the algebracic dual of $V$.
The tensor product $\X\ot\Y$ will be constructed as $B(\X\times\Y)^\#$, by defining the action of $x\ot y$ on $B(\X\times\Y)^\#$ as
$$x\ot y(A) = A(x,y),  \ \  x \in \X, y\in \Y.$$
Thus, a typical tensor $u \in \X\ot \Y$, has the form 
\begin{align}
 \label{typtensor}
u = \sum_{i=1}^I x_i\ot y_i.
\end{align}
Notice that by definition $u = 0$,  if 
$$\sum_{i=1}^{I}A(x_i,y_i) = 0 , \  \mbox{ for all } \ A \in B(\X\times\Y).$$
The following theorem gives an easy criterion to check if $u=0$.
\begin{theorem}
Let $u$ be a tensor of the form (\ref{typtensor}). Then $u=0$ if and only if 
$$\sum_{i=1}^I \phi(x_i)\psi(y_i)=0, \ \ \mbox{ for all } \phi \in \X^*, \psi \in \Y^*.$$
\end{theorem}

So far we have introduced tensor product $\X\ot\Y$ as a vector space. Many choices of norm exist to complete the  space $\X\ot\Y$, e.g the projective norm, the nuclear norm etc. Here however, we focus on the case when $\X$ and $\Y$ are separable Hilbert spaces and the norm considered on $\X\ot\Y$ is Hilbert-Schmidt.

\subsection{Hilbert-Schmidt operator and tensor product}
\label{HSnorm}
Let $\X$ and $\Y$ be two separable Hilbert spaces. Let $\{e_j\}$ be a complete orthonormal system of $\X$. $S \in L(\X,\Y)$ is a Hilbert-Schmidt operator if 
$$\sum_{j}\|Se_j\|^2_{\K} < \infty.$$
The quantity in the left side does not depend on the orthonormal system $\{e_j\}$, and its square root is defined as the Hilbert-Schmidt norm $\|S\|_{HS}$. The space of all Hilbert-Schmidt operators is denoted by $HS(\X,\Y)$. $HS(\X,\Y)$ is a separable Hilbert space.\\

Let $h, h' \in \X$ and  $k,k' \in \Y$. Define an inner product $\<\cdot,\cdot\>_{HS}$ on $\X\ot \Y$ by
$$\<h\ot k, h'\ot k'\>_{HS} = \<h,h'\>_{\X} \ \<k,k,\>_{\Y}.$$
Let $\X\hat{\ot}_{HS}\Y$ denote the completion of the space with respect to the inner product $\<\cdot,\cdot\>_{HS}$. Then $\X\hat{\ot}_{HS}\Y$ is isometrically isomorphic to $HS(\Y,\X)$ and also $HS(\X,\Y)$.\\
If $\{e_j\}$ and $\{f_j\}$ are complete orthonormal systems of $\X$ and $\Y$, then $\{e_j\ot f_k\}_{j,k}$ forms a complete orthonormal system of $\X\hat{\ot}_{HS}\Y.$
If $T \in HS(\X,\Y)$, then $T$ can be represnted as
$$ T  = \sum_{j,k} \<Te_j, f_k\> e_j\ot f_k.$$

\subsection{Infinite-dimensional It\^o's lemma}
\begin{theorem}\label{infito} \citep[Theorem 27.2]{Mm82}
Let $\X$ and $\Y$ be two separable Hilbert spaces, $Z$ an adapted $\X$-valued semimartingale and $\phi:\X\rt \Y$ be a twice continuously differentiable  function with first and second-order derivatives denoted by $D\phi$ and $D^2\phi$ respectively. Assume that for each $x\in \X$, $D^2\phi(x)$ is an element of $L(\X\oth\X,\Y)$ and the mapping $x\rt D^2\phi(x)$ is  uniformly continuous on any bounded subset of $\X$. Then
\begin{align*}
\phi(Z_t) &= \phi(Z_0) + \int_0^t D\phi(Z(s-)) \ dZ(s) +\f{1}{2}\int_0^t D^2\phi(Z(s-)) \ d[Z,Z]^\ot(s) \\
& \hspace{.5cm}+ \sum_{s\leq t}\le(\phi(Z(s))- \phi(Z(s-)) - D\phi(Z(s-))\Delta Z(s) - \f{1}{2}D^2\phi(Z(s-))\Delta Z(s)\ot\Delta Z(s)\ri)\\
& =   \phi(Z_0) + \int_0^t D\phi(Z(s-)) \ dZ(s) +\f{1}{2}\int_0^t D^2\phi(Z(s-)) \ d[Z,Z]^{c,\ot}_s \\
& \hspace{.5cm}+ \sum_{s\leq t}\le(\phi(Z(s))- \phi(Z(s-)) - D\phi(Z(s-))\Delta Z(s) \ri)
\end{align*}
where $[Z,Z]^{c,\ot}_t = [Z,Z]^\ot_t - \sum_{s\leq t}\Delta Z(s)\ot\Delta Z(s).$
\end{theorem}

\subsection{Proof of \eqref{tildefunct}}\label{prf}
Let $\{\gamma_k\}$ be an orthonormal basis of the Hilbert space $\H$. Then a basis for $\R\times\H\times \H$ is given by
$\{e^1=(1,0,0)^T, e^2_i=(0,\g_i,0)^T, e^3_i=(0,0,\g_i)^T:i=1,2,\hdots\}.$ Consequently, a basis for $HS(\R\times\H\times \H,\R\times\H\times \H)$ is given by
$\{e^1\ot e^{k}_i, e^k_i\ot e^1, e^k_i\ot e^l_j:k,l=2,3, \  i,j=1,2,\hdots\}$. Now an expansion of $\Xi \in HS(\R\times\H\times \H,\R\times\H\times \H) $ gives
\begin{align}
\non
\Xi& = \beta e^1\ot e^1 + \sum_i \<h_{12},\g_i\>_\H e^1\ot e^2_i +  \sum_i \<h_{13},\g_i\>_\H e^1\ot e^3_i \\ 
\non
& \quad + \sum_i \<h_{21},\g_i\>_\H e^2_i\ot e^1+  \sum_{i,j} \<\xi_{22},\g_i\ot \g_j\>_{\H\oth\H}e^2_i\ot e^2_j+  \sum_{i,j} \<\xi_{23},\g_i\ot \g_j\>_{\H\oth\H}e^2_i\ot e^3_j \\
& \quad +  \sum_i \<h_{31},\g_i\>_\H e^3_i\ot e^1+  \sum_{i,j} \<\xi_{32},\g_i\ot \g_j\>_{\H\oth\H}e^3_i\ot e^2_j+  \sum_{i,j} \<\xi_{33},\g_i\ot \g_j\>_{\H\oth\H}e^3_i\ot e^3_j.
\label{expansion}
\end{align}
Observe that
\begin{align*}
F(\tilde{x})(e^1)=(1,0,0)^T, \ F(\tilde{x})(e^2_i)= (0, \<\s,\g_i\>_\H,0), \  \ F(\tilde{x})(e^3_i)= (0, 0, \g_i)
\end{align*}
By the definition of the mapping \ \ $\tilde{}$\ \  in Lemma \ref{tildemap}, and using \eqref{opF} and \eqref{opDF}
\begin{align*}
\tilde{DF(\tilde{x})F(\tilde{x})}(e^2_i\ot e^2_j)& = \le(\tilde{DF(\tilde{x})F(\tilde{x})}(e^2_j)\ri)(e^2_i)=
\begin{pmatrix}
0 & 0 & 0\\
0 & \<\s,\g_j\>_\H \partial_2 \s& 0\\
0 & 0 & 0
\end{pmatrix}e^2_i \\
& = (0, \<\s,\g_j\>_\H \<\partial_2 \s,\g_i\>_\H, 0) = (0, \<\partial_2 \s \ot \s,\g_i\ot\g_j\>_{\H\oth\H}, 0). 
\end{align*}
Similarly,
\begin{align*}
\tilde{DF(\tilde{x})F(\tilde{x})}(e^2_i\ot e^3_j)& = \le(\tilde{DF(\tilde{x})F(\tilde{x})}(e^3_j)\ri)(e^2_i)=
\begin{pmatrix}
0 & 0 & 0\\
0 & \partial_3\s \g_j & 0\\
0 & 0 & 0
\end{pmatrix}e^2_i \\
& = (0, \<\partial_3\s \g_j, \g_i\>_\H, 0)^T = (0, \<\partial_3\s, \g_i \ot \g_j\>_{\H\oth\H}, 0)^T. 
\end{align*}
and
\begin{align*}
\tilde{DF(\tilde{x})F(\tilde{x})}(e^2_i\ot e^1)& = \le(\tilde{DF(\tilde{x})F(\tilde{x})}(e^1)\ri)(e^2_i)= (0,\<\partial_1\s,\g_i\>_\H, 0)^T.
\end{align*}
It can easily be checked that other terms are $(0,0,0)^T$. \eqref{tildefunct} now follows from the expansion \eqref{expansion}.

\vspace{1cm}
\np
\textbf{Acknowledgement:}
It is a pleasure to thank  Professor Tom Kurtz for his numerous advice and comments throughout the preparation of the paper.

\bibliographystyle{plainnat}
\bibliography{ref}
\end{document}